\documentclass[11pt]{amsart}
\usepackage[utf8]{inputenc}  
\usepackage{graphicx,wrapfig} 
\usepackage{setspace}         
\usepackage[margin = 1in]{geometry}  
\usepackage[mathscr]{euscript}       
\usepackage{csquotes}                
\usepackage{bm}                      
\usepackage{mathrsfs}
\usepackage{appendix}                
\usepackage{subcaption}              
\usepackage{verbatim}
\usepackage{mathtools}
\usepackage{bbm}
\usepackage[algoruled,linesnumbered]{algorithm2e}
\usepackage{xpatch}
\usepackage{wrapfig}         
\usepackage{amsthm,amsmath,amssymb}
\usepackage{enumitem}
\usepackage{pifont}
\usepackage{centernot}

\usepackage{hyperref}
\usepackage[capitalize]{cleveref}
\usepackage{xcolor}
\usepackage{soul}
\usepackage[colorinlistoftodos]{todonotes}  
\newcommand{\revised}[1]{\textcolor{black}{{#1}}}
\newcommand{\revision}[1]{\textcolor{black}{{#1}}}
\newcommand{\rev}[1]{\textcolor{black}{{#1}}}
\theoremstyle{definition}
\newtheorem{definition}{Definition}[section]

\newtheorem{remark}{Remark}
\newtheorem{lemma}{Lemma}
\newtheorem{proposition}{Proposition}

\newtheorem{theorem}{Theorem}

\newcommand{\tens}[1]{\bm{\mathcal{#1}}}
\newcommand{\mat}[1]{\bm{#1}}

\def\tC{{\tens{C}}}

\def\tQ{{\tens{Q}}}
\def\tR{{\tens{R}}}

\def\tU{{\tens{U}}}

\def\tX{{\tens{X}}}  
\def\tY{{\tens{Y}}}
\def\tZ{{\tens{Z}}}

\def\va{{\bm{a}}}
\def\vb{{\bm{b}}}

\def\vi{{\bm{i}}}

\def\vr{{\bm{r}}}
\def\vs{{\bm{s}}}

\def\vu{{\bm{u}}}
\def\vv{{\bm{v}}}

\def\vx{{\bm{x}}}
\def\vy{{\bm{y}}}
\def\vz{{\bm{z}}}

\newcommand{\kron}{\otimes}
\newcommand{\out}{~\raisebox{-1pt}{\tikz \draw[line width=0.5pt] circle(3.9pt);}~}

\newcommand{\algoname}[1]{\textnormal{\textsc{#1}}}

\newcommand{\R}{\mathbb{R}}

\newcommand{\N}{\mathbb{N}}

\newcommand{\E}{\mathbb{E}}

\graphicspath{ {./images/} } 
\usepackage[backend = biber, style = numeric, sorting = ynt]{biblatex}
\title{\vspace{-10mm}{On trimming tensor-structured measurements and \\ efficient low-rank tensor recovery 
}}

\author{Shambhavi Suryanarayanan, Elizaveta Rebrova}

\begin{document}

\begin{abstract}
In this paper, we take a step towards developing efficient hard thresholding methods for low-rank tensor recovery from linear measurements with tensorial structure. Theoretical guarantees for many standard iterative low-rank recovery methods, such as iterative hard thresholding (IHT), are based on model assumptions on the measurement operator, like the restricted isometry property (RIP). However, tensor-structured random linear maps -- while memory-efficient and convenient to apply -- lack good restricted isometry properties; that is, they do not preserve the norms of low-rank tensors sufficiently well. To address this, we propose local trimming techniques that provably restore point-wise geometry-preservation properties of tensor-structured maps, making them comparable to those of unstructured linear measurements. 

Then, we propose two novel versions of tensor IHT algorithms: an adaptive gradient trimming algorithm and a randomized Kaczmarz-based IHT algorithm that efficiently recover low-rank tensors from linear measurements. We provide initial theoretical guarantees for the proposed methods and present numerical experiments on real and synthetic data, highlighting their efficiency over the original TIHT for low HOSVD- and CP-rank tensors.


\end{abstract}
\maketitle

\section{Introduction}

Tensors, as multi-modal arrays\revision{,} are natural choices for analyzing realistic, high-dimensional structures in \revision{various} application areas. They have been widely used in recent years for modeling objects in signal processing, medical imaging, machine learning, and other domains \cite{sidiropoulos2017tensor,cichocki2015tensor,zhou2016linked, hunyadi2017tensor, guo2016support}. The ubiquity of their applications has motivated the development of specialized techniques to efficiently process large-scale multi-modal data \cite{liu2012tensor,zhang2016exact,papalexakis2016tensors}. Due to the large scale of many applications, it is not surprising that the cornerstone techniques in this suite are related to compression and subsequent recovery of tensorial data \cite{ballester2019tthresh, austin2016parallel}. 

Data-oblivious random sketching is a \revision{fundamental} and powerful linear approach for taming large-scale data. For a given large instance of the data $\tX \in \R^{N}$, the goal is to replace $\tX$ with $\mathcal{A}\tX$ where $\mathcal{A} \in \R^{m \times N}$ is a random matrix with $m \ll N$. Unlike data-aware low-parametric fitting (such as low-rank fitting)\revision{,} which aims to find the best possible approximation for \revision{a single} large and high-dimensional object, selecting a random operator $\mathcal{A}$ \revision{enables us to obtain} uniform guarantees for many such large data instances\revision{,} all while avoiding costly fitting procedures. For example, when it is crucial to guarantee the validity of a data sketch through an iterative process, one requires uniform guarantees over all possible iterates.

A key question in oblivious linear dimensionality reduction is how to design memory-efficient random measurement maps $\mathcal{A}$ that can be applied to the data fast and lead to good compression with high probability \cite{kane2014sparser, dasgupta2010sparse}. Memory efficiency is especially important for tensorial data, since vectorizing a tensor of dimension $n$ in $d$ modes leads to an object in $\R^N$ with $N = n^d$, and sketching matrices have to be correspondingly very large-scale.
Recent work has studied tensor-specific, fast, and memory-efficient sketching operators. Some key examples include Kronecker (or, modewise) measurements, \emph{face-splitting} measurements (that consist of independent scalar Kronecker-structured measurements), and more sophisticated tensor-structured measurements are usually built upon the former (an incomplete list includes \cite{sun2021tensor, ahle2020oblivious, jin2021faster}).
In addition to requiring fewer random bits to store structured matrices, these strategies allow fast application to tensor data by mimicking its structure.

The quality of data-oblivious dimension reduction depends on the geometry-preserving properties of the measurement matrices\revision{,} which are often chosen at random. When a linear map satisfies Johnson-Lindenstrauss-type properties, it can be used to reduce dimensionality and speed up a variety of tasks such as clustering and regression \cite{ailon2009fast, cohen2015dimensionality, wang2018sketched}. When the goal is the subsequent recovery of the data, the map is usually required to satisfy the Restricted Isometry Property (RIP, see Definition~\ref{def:rip} below), which is a uniform guarantee for approximate norm preservation over infinite data sets, such as sparse vectors or low-rank matrices \cite{candes2008restricted}. \revised{For example, RIP is required for successful convergence on popular iterative recovery algorithms such as Iterative Hard Thresholding (IHT, \cite{blumensath2009iterative}), among others \cite{foucart2011hard, needell2010signal, needell2009cosamp}.}

\revised{Here, we are interested in recovery of low-rank tensors under the standard notions of tensor rank, specifically, CP and HOSVD, from data-oblivious linear measurements.} While RIP properties of generic i.i.d. matrices extend to the low-rank tensor case for various tensor ranks (e.g., \cite{rauhut2017low, grotheer2023iterative}), this is not the situation for memory-efficient measurements, such as \revised{Kronecker or face-splitting.} 
So, typically, memory-efficient measurements are used within tasks and algorithms that do not require uniform control over the iterates (and where weaker guarantees for the norm preservation suffice). \revised{This includes multiple works on the recovery of low-HOSVD rank tensors from streaming and memory-efficient measurements, e.g.,  \cite{iwen2021lower,haselby2025fast,sun2021tensor, che2025efficient}. These works rely heavily \revision{on} the structure of HOSVD decomposition (see \eqref{eq:hosvd} below) and show that the recovery is possible with high probability for a fixed (sampled) measurement matrix. In contrast, the uniform recovery guarantees -- such as those based on TensorRIP \revised{(see \cref{def:tensorrip} below)} -- allow \revision{for} provably \revision{reusing} the same measurement matrices for all low-rank tensors of interest.} A  proposed workaround to get TensorRIP  for Kronecker measurements was based on a partial reshaping of the tensor by flattening several modes together \cite{haselby2023modewise}. Such a reshaping allowed the  Tensor IHT algorithm to successfully recover tensors from memory-efficient measurements. However, the reshaping approach fails to give nontrivial results in an important case of a $3$-mode tensor, as the only option is to flatten all three modes together\revision{,} thus rendering the \revision{measurement no longer memory efficient}. \emph{An important motivation for the current work is to give the first nontrivial IHT-like algorithm for 
$3$-d tensor low-rank recovery from memory-efficient measurements.} 

In this work, we focus on the following questions: \emph{(1) To what extent are the geometry preservation properties of tensorial memory-efficient measurements weaker than \revision{those} of the standard sub-Gaussian measurement matrices? (2) How can we efficiently run data-oblivious recovery procedures similar to the IHT algorithm for recovery from tensorial measurements?} 

Our main contributions are two-fold:
\begin{itemize}
    \item \emph{RIP-like properties of  face-splitting measurements under trimming.} We study the limitations of memory-efficient face-splitting measurements as restricted isometry maps and propose adaptive trimming-based local corrections that improve the geometry preservation properties of such operators.
    \item \emph{Efficient iterative hard thresholding variants for low-rank tensor recovery from generic linear measurements.} Motivated by this approach, we propose two new versions of the iterative hard thresholding algorithm for low-rank tensor recovery, TrimTIHT and KaczTIHT, that empirically demonstrate the improved performance for  i.i.d. measurements, especially in the hard case of tensor-structured memory-efficient measurements. We also provide convergence theory for the proposed methods, giving them initial theoretical validation. 
\end{itemize}

\subsection{Paper structure and roadmap.} This paper has the following structure. In Section~\ref{sec:prelim}, we review tensor prerequisites and definitions, tensor low-rank recovery problems\revision{,} and their challenges. In Section~\ref{sec:facespkit}, we focus on memory-efficient tensorial measurements with the face-splitting structure, proving the following

\begin{theorem}[Theorem~\ref{theorem:trim_trip} and Proposition~\ref{prop:db_friendly_trip}, informal version] \label{thm:main-informal}\footnote{\textit{Revision note}. In the previous version, the amount of rows that needed to be removed was stated incorrectly. The corrected result requires removing a constant fraction of the rows, rather than a logarithm. Accordingly, this suggests the trimming of a constant fraction of rows of $\mat{A}$ in the TrimTIHT algorithm. All other conclusions and derivations, as well as the numerical results, remain unchanged.}
Let $\mat{A} \in \R^{m \times n^d}$ be a matrix with the rows given by Kronecker products of $d$ random vectors in $\R^n$ with i.i.d. mean zero, variance one, sub-Gaussian entries.
Then, for the set $\mathcal{S}$ of all unit norm tensors of HOSVD rank at most $(r, \ldots, r)$, we have with probability $1 - c/n$, 
 \begin{equation}\label{eq:trim-rip}
     \sup_{\tY \in \mathcal{S}}
    \left  \vert \|\mat{A}^\vy\vy\|_2^2 - 1 \right \vert \leq \delta\;,
 \end{equation}
as long as $m \gtrsim \delta^{-2}d (r^d + dnr){\ln (n/\delta)}$\footnote{For comparison, i.i.d. sub-Gaussian measurements are known to satisfy \eqref{eq:trim-rip} -- without any trimming --  for $m$ of the same order up to the logarithms and an additional $d$-factor, see further discussion in Section~\ref{sec:prelim}.}. Here, $\mat{A}^\vy$ is a version of $\mat{A}$ without \rev{a suitable fraction of} rows that are most aligned with $\vy = \algoname{Vec}(\tY)$. If $\mathcal{S}$ denotes the set of all unit norm tensors of CP-rank at most $r$, the same holds for $m \gtrsim \delta^{-2} d^2nr{\ln (n/\delta)}$. However, if we do \emph{not} remove any rows, that is, we try to get a standard RIP guarantee for the set low-rank tensors, then we fail essentially with any nontrivial amount of compression: for any $\alpha > 0$, the probability of \eqref{eq:trim-rip} with $\mat{A}^\vy = \mat{A}$ goes to zero for any $m < n^{d(1 - \alpha)}$. 
\end{theorem} 

Then, in Section~\ref{sec:iht-algorithms}, we propose new variants of \revision{the} TIHT algorithm that are applicable to generic, and not necessarily face-splitting, linear measurements.
Informally, the original IHT algorithm alternates gradient steps with sparse (or low-rank) fitting. In \cref{alg:trimtiht}, TrimTIHT, 
we replace the gradient step of the TIHT algorithm with a scaled gradient estimate calculated based on a subset of rows of $\mat{A}$, chosen adaptively to the current iterate\revision{,} as to enhance its norm preservation properties. To analyze the convergence of the method, we develop a path-based analysis of the convergence of TIHT, see Theorem~\ref{thm:main_path}. Then, in Remark~\ref{rem:trimtiht-discussion}, we discuss why our trimming results, such as Theorem~\ref{thm:main-informal}, in particular, suggest the advantages of TrimTIHT on the face-splitting measurements.  

The idea of the second proposed algorithm, KaczTIHT (Algorithm~\ref{alg:kziht}), is to substitute the gradient step with a sequence of stochastic gradient steps, thereby employing the convergence properties of the Kaczmarz linear solver \cite{karczmarz1937angenaherte, strohmer2009randomized}. This is a multi-order extension of the recently proposed \cite{zhang2015iterative} and analyzed \cite{jeong2025linear} Kaczmarz iterative hard thresholding method. We also provide a convergence guarantee for the KaczTIHT method, however under the stronger assumptions of i.i.d. sub-Gaussian measurements.

Then, in Section \ref{section:experiments}, we empirically demonstrate the advantages of the proposed methods, which are evident across a range of settings and are especially pronounced in more challenging recovery scenarios—such as when tensors have higher ranks or are compressed using less regular measurements (e.g., memory-efficient face-splitting measurements). We conclude and discuss the future directions of this work in Section~\ref{sec:conclusions}.

\section{Preliminaries and Related Work}\label{sec:prelim}
Throughout this paper, we adopt the following notations: the bold calligraphic font (e.g. $\tX,\tY$) is dedicated solely to tensors, the capital Roman script (e.g. $\mat{X},\mat{Y}$) is for matrices and the bold lower case (e.g $\vx,\vy$) is for vectors. The element of the tensor $\tX$ corresponding to the multi-index $\vi = (i_1, i_2, \dots i_d) \in [n_1] \times [n_2] \dots [n_d]$, is denoted by either $\tX(\bm{i})$ or $\tX_{i_1i_2\dots i_d}$. For two real functions $f,g$, we say that $f = O(g)$ if $|f(x)| \leq M|g(x)|$ for all $x>c$, for some $M$ and $c \in \R^+$.  We will also use $\|\mat{A}\|_2$ to denote the operator norm of a matrix $\mat{A} \in \R^{m \times n}$ that is given by $\sup_{\vv \in \mathbbm{S}^{n-1}} {\|\mat{A}\vv\|_2}$. Here and further, let $\mathbbm{S}^{n-1}$ denote the unit sphere in $\R^n$.

\subsection{Tensor preliminaries}\label{sec:prelims}
For an in-depth overview of tensor-related concepts, we refer to \cite{ballardtensor, kolda2009tensor}. Here, we include and briefly discuss the notions crucial for this paper.

\underline{\emph{Reshaping.}} By collapsing the tensor along all modes, we can map the tensor $\tX~\in~\R^{n_1 \times n_2 \dots \revised{\times} n_d}$ to a vector $\vx \in \R^{n_1 n_2 \dots n_d}$.  
\revised{Specifically, we define a map $\algoname{Vec}: \R^{n_1 \times n_2 \dots \times n_d} \rightarrow \R^{n_1n_2 \dots n_d}$ by
$$   \tX_{i_1, i_2, \dots, i_d} \mapsto \vx_i = \algoname{Vec}(\tX)_i \; \text{ where } i = 1 + \sum_{j=1}^{d} (i_j - 1) \prod_{k = j+1}^d n_k.$$}
 \underline{\emph{Inner Product and Norm.}} The inner product and Frobenius norm of the tensors $\tX, \tY \in \R^{n_1 \times n_2 \dots \times  n_d}$ are defined as the inner product and Euclidean norm of their respective vectorizations:
 \begin{gather*} 
 \langle \tX, \tY \rangle  = \langle \algoname{Vec}(\tX), \algoname{Vec}(\tY)\rangle \text{ and } \quad \|\tX\|_F = \sqrt{\langle \tX, \tX\rangle} = \|\algoname{Vec}(\tX)\|_2\;.
 \end{gather*}

\underline{\emph{Modewise product.}} A key operation on tensors is their modewise multiplication with matrices. Given a tensor $ \tX \in \R^{n_1 \times n_2 \dots \times n_d}$, its $k$-mode tensor product with a matrix $\mat{A} \in \R^{m \times n_k}$, denoted by $\tX \times_{k} \mat{A} \in \R^{n_1 \times \dots n_k \times m \times n_{k+1} \dots n_{d}}$, is given entry-wise by 
\[ (\tX \times_k \mat{A})_{i_1 i_2 \dots i_k \dots i_d} = \sum_{j = 1}^{n_k} \tX_{i_1 \dots i_{k-1} j i_{k+1} \dots i_d} \mat{A}_{i_k j}\;. \]

\underline{\emph{Rank-1 tensor.}} $\tX \in \R^{n_1 \times n_2 \dots \times n_d}$ is said to be a \emph{rank one} tensor if it can be expressed as 
\begin{equation}\label{eq:rank-1-def} \ \  \tX = \vx_1 \out \vx_2 \dots \out\vx_d, \quad \text{ or } \quad (\tX)_{i_1 i_2 \dots i_k \dots i_d} = (\vx_1)_{i_1} (\vx_2)_{i_2} \dots (\vx_d)_{i_d}, \end{equation}
where $\out$ denotes vector outer product and the vectors $\vx_i \in \R^{n_i}$ for $i = 1, \ldots, d$. If $\tX$ is a rank-1 tensor, the modewise product of $\tX$ with any matrix $\mat{A} \in \R^{m \times n_k}$ can be elegantly written in terms of its components $\vx_i$ as-
$$ \tX \times_k \mat{A} = \left(\out_{i=1}^{k-1} \vx_i \right)\out \mat{A} \vx_k \out \left(\out_{i=k+1}^{d} \vx_i\right). $$

\underline{\emph{Tensor Decompositions and Tensor Rank.}} Extending the definition of tensor rank from rank 1 to rank $k$ for $k > 1$ has significant difficulties, resulting in multiple definitions of tensor rank. Two of the most natural notions of tensor rank that are considered in this paper are 
\begin{enumerate}
    \item \textbf{[CP]} A tensor $\tX \in \R^{n_1 \times n_2 \dots n_d}$ is said to be of \emph{{CP rank}}-$r$ if it can be expressed as a sum of $r$ rank-1 tensors. That is, if $\exists \ \vx_{ij} \in \R^{n_i}$ for all $i \in [d]$ and $j \in [r]$ such that
\[ \tX= \sum_{j=1}^{r} \vx_{1j} \out \vx_{2j} \dots \out \vx_{dj}\;. \]
 While this decomposition of a tensor into vectors is very convenient, such \revision{a} minimal $r$ is NP-hard to find or even verify for a given tensor. \revision{Additionally}, one cannot guarantee even approximate orthogonality of the factor vectors. This motivates alternate definitions of the tensor rank.

\item \textbf{[HOSVD]} Given a tensor $\tX\in \R^{n_1 \times \dots \times n_d}$, the \emph{Tucker decomposition} of $\tX$ consists of a core tensor $\tC \in \R^{r_1 \times \dots \times r_d}$ and matrices $\mat{U}_i \in \R^{n_i \times r_i}$ for $i \in [d]$ where
\begin{equation}\label{eq:hosvd} \tX = \tC \times_1 \mat{U}_1 \times_2 \dots \times_d \mat{U}_d\;. \end{equation}
If for all $i \in [d]$, the columns of the matrix $\mat{U}_i$ form an orthonormal basis, this decomposition is known as the \emph{Higher-Order Singular Value Decomposition}. The HOSVD (or Tucker) rank of $\tX$ is given by $\vr = (r_1, r_2, \dots, r_d)$, the dimension of the core tensor $\tC$. Furthermore, the HOSVD decomposition can be chosen such that the core-tensor has orthogonal sub-tensors. More formally, let $\tQ_{k_i=p} \in \R^{n_1 \times n_2 \dots n_{i-1} \times n_{i+1} \dots n_d}$ denote the $(d-1)$-order sub-tensor obtained by fixing the $i$-mode coordinate to $p$ for some $i \in [d]$. Then, for $p \neq q$,
    $$ \langle \tQ_{k_i = p}, \tQ_{k_i = q} \rangle = 0\;.$$
\end{enumerate}
Various domain-specific applications have motivated the definition of different notions of tensor rank that are beyond the scope of this paper, see, e.g., \cite{ballardtensor} \revision{for} further discussion.

 \underline{\emph{Low-Rank Tensor Fitting.}} Unlike the matrix setting, where the Eckart-Young theorem gives us a direct way to obtain the best rank-$k$ approximation using SVD \cite{eckart1936approximation}, such guarantees are not as straightforward for higher-order tensors. In fact, under the CP setting, the best rank $k$-approximation might not exist, and even when it does, \revision{it} is NP-hard to compute \cite{paatero2000construction}. 
 
 In \cite{song2019relative}, the authors proposed randomized algorithms for low-rank fitting that are poly-time in the dimensions of the tensor as follows. For simplicity, let us state the results for the tensors with the same dimension in each mode.
 \begin{theorem}[Low CP Rank Approximation \protect{\cite[Theorem 1.2]{song2019relative}}]
\label{thm:opt_cp_approx}
 Given an order $d$ tensor  $\tX$ in $\R^{ n \times n \dots \times n}$ and $\varepsilon > 0$ , let $$\algoname{Opt} = \underset{\tY - \text{CP rank r}}{\min}\|\tX - \tY\|^2_F.$$ If there exists $\delta > 0$ and
$\tX_r~=~\sum_{j=1}^{r}~\vx_{1j}~\out~\vx_{2j}~\dots~\out~\vx_{dj}$ such that $\|\vx_{ij}\| \leq 2^{O(n^\delta)}$ for all $i \in [d]$ and $j \in [r] $ and $\|\tX - \tX_r\|_F^2 \leq \algoname{Opt} + 2^{-O(n^\delta)}$, then a proposed algorithm outputs a rank $r$-tensor $\tZ$ such that
$$ \|\tZ - \tX\|^2_F \leq (1+\varepsilon)\algoname{Opt} + 2^{-O(n^\delta)}\;.$$
Further, this algorithm runs in $O\left(n^{d+\delta} +n^{\delta+1} \text{poly}(r,1/\varepsilon)+n^{\delta}2^{O\left(\text{poly}(r, 1/\varepsilon)\right)}\right)$ time. 
 \end{theorem}

 \begin{theorem}[Low HOSVD Rank Approximation \protect{\cite[Theorem L.2]{song2019relative}}]
 \label{thm:opt_hosvd_approx}

 Given an order $d$ tensor $\tX \in \R^{n \times n \times \dots \times n}$ and $\varepsilon > 0$, there exists an algorithm that outputs a HOSVD rank $\vr = (r,r, \dots ,r)$ tensor $\tZ$ such that, 
$$ \| \tZ - \tX \|_F^2 \leq (1+\varepsilon) \| \tX - \tX_{best} \|_F^2\;,$$
where $\tX_{best}$ is the best HOSVD-rank $\vr$ approximation of $\tX$. Further, the running time of this algorithm is  $O\left(n^d +n \text{poly}(r, 1/\varepsilon)+2^{O\left(\text{poly}(r, 1/\varepsilon)\right)} \right)$.
 \end{theorem}
 These algorithms, while providing good approximation guarantees, \revision{require additional} assumptions for \revision{the} CP rank and are computationally expensive to run. In the standard tensor packages, variants of the alternating least squares algorithm are used to compute a CP rank-$k$ approximation to any given tensor \cite{battaglino2018practical}.  Similarly, for low-rank Tucker approximation, the commonly used algorithms include the Higher Order Orthogonal Iterations (HOOI), HOSVD\revision{,} and ST-HOSVD \cite{de2000multilinear,vannieuwenhoven2012new}. Both HOSVD and ST-HOSVD are $\sqrt{d}$-quasi-optimal. That, is given a tensor $\tX$, the rank-$\vr$ output of these algorithms, $\bar{{\tX}}$ satisfies - 
\begin{equation}\label{eq:fitting} \| \tX - \bar{\tX} \|_F \leq \sqrt{d} \| \tX - \tX_{best} \|_F\;,
\end{equation}
where $\tX_{best}$ is the optimal rank-$\vr$ approximation. The HOOI algorithm often gives better approximation guarantees in practice. However, it can be accompanied by increased computational and storage costs \cite{de2000best} depending on the number of iterations. For the experiments described in this paper, we use CP-ALS and HOOI from the Tensorly package \cite{kossaifi2019tensorly} in Python for low-rank CP and Tucker approximation\revision{,} respectively. 

It is pertinent to note that results such as \eqref{eq:fitting} are worst-case bounds\revision{,} and in practice, these methods provide better approximations for most instances. This gap between theory and observations led authors to assume a better bound on this approximation in some works, e.g., \cite{rauhut2017low} (see more detailed discussion in Section~\ref{sec:iht} below). 
In our theoretical analysis, we will adopt similar assumptions for the low-rank approximation operator.

\subsection{RIP and geometry preserving linear operators}

A standard way to capture the ability of some operator $\mathcal{A}$ \revision{in preserving} geometric properties of the objects from a class is through the restricted isometry property: 
\begin{definition}\label{def:rip}[$(\delta,\mathcal{S})$-Restricted Isometry Property (RIP)] A measurement operator $\mathcal{A}: \mathcal{S} \rightarrow \R^m$ satisfies  \emph{$(\delta,{\mathcal{S})}$}-RIP if
$$ (1 - \delta)\|\vs\|^2 \leq  \|\mathcal{A}(\vs)\|^2 \leq (1+\delta) \|\vs\|^2 \quad \forall \vs \in \mathcal{S}.$$    
\end{definition}
Informally, RIP guarantees that the sketching with $\mathcal{A}$ approximately preserves the norm of all objects within the set of interest $\mathcal{S}$. Depending on the task, $\mathcal{S}$ can be an arbitrary finite set (then Definition~\ref{def:rip} reduces to the classical Johnson-Lindenstrauss property), or an unknown low-dimensional subspace (then it is known as the Oblivious Subspace Embedding property), or any collection of objects of low-complexity, such as, all sparse vectors or all matrices of low-rank $r$. The latter case is the standard interpretation of the RIP property. Here, we are interested in the sets of all low-rank tensors, so we call the corresponding property TensorRIP:
\begin{definition}[TensorRIP]\label{def:tensorrip}
Let \revised{us} denote the sets of low-rank tensors as follows:
\begin{gather*}
\mathcal{S}_\vr^{HOSVD} =\left\{\tX \in \R^{n_1 \times n_2 \dots \times n_d} \ \vert \ \text{HOSVD Rank}(\tX) \revised{\text{ at most }} \vr = (r_1, r_2 \dots r_d), \ \|\tX\|_F = 1 \right\}, \\
\mathcal{S}_r^{CP} =\left\{\tX \in \R^{n_1 \times n_2 \dots n_d} \ \vert \ \text{CP Rank}(\tX) \revised{\text{ at most }} r , \ \|\tX\|_F = 1 \right\}.
\end{gather*}
\revised{We define $(\delta, \mathbf{r})$-TensorRIP as $(\delta, \mathcal{S})$-RIP with $\mathcal{S} = \mathcal{S}_\vr^{HOSVD}$. Similarly, $(\delta, r)$-TensorRIP denotes $(\delta, \mathcal{S}_r^{CP})$-RIP property.} 
\end{definition}
\revised{In this paper, we will focus on linear measurement operators. Given a matrix $\mat{A} \in \R^{m\times n_1n_2\dots n_d}$, we can then define a linear operator, $\mathcal{A}$: $\R^{n_1 \times n_2 \dots \times n_d} \to \R^m$, as} $\mathcal{A}(\tX) = \mat{A} \cdot\algoname{Vec}(\tX)$. \revised{In the following sections, we will interchangeably represent this operator $\mathcal{A}$ acting on the tensors by its linear measurement matrix $\mat{A}$ that acts on the vectorization. In particular, we say that a matrix $\mat{A}$ satisfies TensorRIP if the corresponding linear operator satisfies it.}

A key \revised{ingredient for establishing} RIP  over some set $\mathcal{S}$ is \revision{to bound} the intrinsic complexity of the set, and one of the standard ways to do so is via covering number\revised{s}:
\begin{definition}[$\varepsilon$-net and Covering Number]
\label{def:epsilon-net}
Consider a subset $K$ of a metric space $(T,d)$, and let $\varepsilon > 0$. Then, a set $\mathcal{N} \subseteq K$ is said to be an $\varepsilon$-net of $K$ if for any $\vx \in K$, there exists $\vx' \in \revised{\mathcal{N}}$ such that $d(\vx,\vx') \leq \varepsilon.$
Further, the covering number of $K$, denoted by $\mathcal{N}(K,d,\varepsilon)$, is the smallest possible cardinality of an $\varepsilon$-net of $K$. 
\end{definition}
\begin{remark}[Covering numbers for low-rank tensors]
\label{lem:covering_numbers_tensors} The following upper bounds for covering numbers for the sets ${\mathcal{S}}_\vr^{HOSVD}$ and ${\mathcal{S}}_r^{CP}$ are known. For any $\varepsilon \in (0,1)$  are given by,
\[   \ln \mathcal{N}({\mathcal{S}}_\vr^{HOSVD}, \|.\|_F, \varepsilon) \leq  ({r_1 r_2\dots r_d + \sum_{i=1}^d n_i r_i}) \ln \left(\frac{3(d+1)}{\varepsilon}\right) \text{\cite{rauhut2017low}},\]
\[ \ln \mathcal{N}({\mathcal{S}}_r^{CP}, \|.\|_F, \varepsilon) \leq C\left(r\sum_{i=1}^d n_i \ln\left(\frac{n_1 n_2 \dots n_d}{\varepsilon}\right) \right) \text{\cite{zhang2025covering}}.\]
for some constant $C > 0$.
The logarithm of the covering numbers $\ln \mathcal{N}(K, \epsilon)$ is often called the metric
entropy of $K$. It is known (e.g., \cite{hdp}) that the metric entropy is equivalent to the number of bits needed to encode points in $K$. Note that the set of $d$-mode HOSVD rank-$\vr$ tensors needs $r_1 r_2\dots r_d + \sum_{i=1}^d n_i r_i$ bits to encode it, so the estimate above is $(\ln d)$-optimal. Similarly, for CP rank-$r$ tensors, $r\sum_{i=1}^dn_i$ bits is enough, and the estimate from \cite{zhang2025covering} is optimal up to an extra $d\ln{n}$ factor. The estimates without this extra \revised{$d$} factor are available (and are much easier to get) under additional conditions like bounded factors or incoherence \cite{grotheer2023iterative,ibrahim2020recoverability}.
\end{remark}

\subsection{Low-rank tensor recovery and Iterative Hard Thresholding algorithm} \label{sec:iht}
The low-rank tensor recovery problem aims to find an unknown $d$-dimensional tensor $\tX \in \R^{n_1 \times n_2 \dots n_d}$ from $m$ possibly noisy measurements $\vb = \mathcal{A}(\tX) + \boldsymbol{\eta} \in \R^m$ of a linear measurement operator $\mathcal{A}: \R^{n_1 \times n_2 \dots n_d} ~\to~\R^{m}$, where typically, $m \ll n_1n_2\dots n_d$, and the noise vector $\boldsymbol{\eta}$ is assumed to have a small norm. 
This is a higher-order extension of a classical problem of sparse vector and low-rank matrix recovery.  The recovery guarantees of structured objects from a small number of linear measurements can be established under some assumptions on the measurement matrix, including RIP \cite{candes2006robust}. Various algorithms have been employed to solve the low-rank matrix recovery problem \cite{chi2016kaczmarz,fornasier2016conjugate,tanner2013normalized,wei2016guarantees}. Some of these methods have been extended to the tensor setting as well \cite{rauhut2017low,luo2023low}. In this paper, we consider a particular example of one of the simplest algorithms, iterative hard thresholding. Iterative thresholding based algorithms are incredibly effective in solving low-rank matrix \cite{cai2010singular} and tensor recovery \cite{grotheer2023iterative,goulart2015iterative} problems. Moreover, the simplicity of this method enables us to modify or adapt it to different settings \cite{grotheer2020stochastic,han2015modified,axiotis2022iterative}.

The tensor iterative hard thresholding (TIHT) algorithm is given by the following iterative steps, starting from some initialization $\tX^{0}$: for $j = 1, 2, \ldots,$
\begin{equation}\label{eq:iht-grad} \tY^{\revised{j}} = \tX^{j} + \mu_j \mathcal{A}^*(\vb - \mathcal{A}(\tX^{j}))\;,
\end{equation}
\begin{equation}\label{eq:iht-fit}\tX^{j+1} = T_{\vr}(\tY^{\revised{j}})\;.
\end{equation}
Here, $T_\vr$ denotes a low-rank approximation operator (which can be realized in \revision{several} ways as we discussed in Section~\ref{sec:prelims}), $\mu_j \in \R^+$ represents the step size\revision{,} and $\mathcal{A}^*$ denotes the adjoint of $\mathcal{A}$. It is known that this algorithm converges under TensorRIP assumptions on the measurement operator $\mathcal{A}$, specifically,
\begin{theorem}[TIHT for Low Rank Recovery \cite{rauhut2017low}]
\label{thm:TIHT_hosvd} Let $\mathcal{A}: \R^{n_1 \times n_2 \dots \times n_d} \rightarrow \R^m$ be a measurement operator satisfying $(\delta,{3\textbf{r}})$-TensorRIP with $ \delta_{3\vr}< \frac{a}{4}$ for some $a \in  (0,1)$. Also, let $\tX \in \R^{n_1 \times n_2 \dots \times n_d}$ be a low rank-$\mathbf{r}$ tensor for some notion of tensor rank and $\vb= \mathcal{A}(\tX) + \bm{\eta}$ be (noisy) measurements of $\tX$. Further, let \revised{us} assume that each iteration of TIHT satisfies 
\begin{equation}
\label{eq:extra_assmpn}
    \|  \tY^j - \tX^{j+1}\|_F \leq \left(1+\frac{a^2}{17(1+\sqrt{1+\delta})^2}\right) \|\tY^j - \tX\|_F.
\end{equation}
 Then the iterates of the TIHT algorithm with the step size one satisfy
$$ \|\tX- \tX^{j+1} \|_F \leq a^j \|\tX - \tX^0\|_F + \frac{\xi(a)}{1-a} \|\bm\eta\|_2 \quad \text { for all } j \in \N,$$
where 
\begin{align*}
\xi(a) = 2\sqrt{1+\delta} + \sqrt{4 \sigma(a) + 2 \sigma(a)^2}\frac{1}{1-\delta}{\|\mat{A}\|_{2}} && \text{and} &&
\sigma(a) = \frac{a^2}{17(1+\sqrt{1+\delta} \|\mat{A}\|_2)^2}. 
\end{align*}
The assumption \eqref{eq:extra_assmpn} on the accuracy of the low-rank approximation operator, stated informally, \revised{ensures that the rank-$\vr$ approximation step \eqref{eq:iht-fit} does not increase the distance to the true solution significantly compared to what it was after the gradient update \eqref{eq:iht-grad}}. It is a stronger assumption than $\sqrt{d}$-order approximation of the standard algorithms for HOSVD fitting, but more flexible than the strong guarantees obtained by the slower theoretical algorithms from Theorem~\ref{thm:opt_hosvd_approx}.

\end{theorem}

\subsection{Independent sub-Gaussian measurement operator and aus}
As in the matrix setting, a large class of random measurement ensembles can be shown to satisfy TensorRIP with high probability. 
\revised{Recall that we identify  the linear measurement operator $\mathcal{A}$ with its transformation matrix $\mat{A}$ that acts on tensor vectorizations.} A common and simple way to generate the random measurement matrices $\mat{A}$ is by independently sampling each of its rows from a sub-Gaussian distribution.

\begin{definition}[Sub-Gaussian Distribution] A random variable $X$ is said to follow a \emph{sub-Gaussian distribution} if there exists some $K > 0$, such that $\mathbbm{E} \exp(X^2/K^2) \leq 2.$ Moreover, the sub-Gaussian norm of $X$, denoted by $\|X\|_{\psi_2}$ is defined to be, 
$$ \|X\|_{\psi_2} = \underset{t>0}{\inf} \ \{\E \exp(X^2/t^2) \rev{ \ \leq 2}\}.$$
\end{definition}
Examples of sub-Gaussian random variables include Gaussian, Bernoulli\revision{,} and all bounded random variables. Compressive measurement operators with independent sub-Gaussian entries are known to satisfy RIP, specifically,
\begin{theorem}[TensorRIP for low HOSVD-rank tensors \protect{\cite[Theorem 2]{rauhut2017low}}] 
\label{lem:trip_hosvd}
Let  $\mat{A} \in \R^{m \times n_1 n_2 \dots n_d}$ be a measurement matrix whose rows consist of randomly sampled independent sub-Gaussian vectors. Then, $\frac{1}{\sqrt{m}}\mat{A}$ satisfies $(\delta,\vr)$-TensorRIP over $\mathcal{S}_\vr^{HOSVD}$ where $ \vr = (r_1,r_2, \dots, r_d)$ (see Definition~\ref{def:tensorrip}), with probability at least $1 - \varepsilon$ if
\begin{equation} m \geq C\delta^{-2} \max\{ (r^d + \sum_i n_ir_i)\ln(d), \ln(\varepsilon^{-1})\}, \end{equation}
\noindent where $r = \underset{i\in[d]}{\max} \ r_i$ and the constant $C>0$ depends on the sub-Gaussian norm of the distribution.
\end{theorem}
We can also establish TensorRIP for low CP rank tensors under sub-Gaussian measurements as a simple corollary to a recent result from \cite{zhang2025covering}. 

\begin{theorem}[TensorRIP for Low Rank CP tensors \cite{zhang2025covering}]
\label{theorem:trip_cp}

Let  $\mat{A} \in \R^{m \times n_1 n_2 \dots n_d}$ be a measurement matrix whose rows consist of randomly sampled independent sub-Gaussian vectors. Then, $\frac{1}{\sqrt{m}}\mat{A}$ satisfies $(\delta,\revised{r})$-TensorRIP over $\mathcal{S}_{r}^{CP}$ (see Definition~\ref{def:tensorrip}) with probability at least $1 - \varepsilon$ if
\begin{equation} m \geq C\delta^{-2} \max\left\{rd^2n\ln(\revised{r}dn),\ln \varepsilon^{-1}\right\},
\end{equation}
\noindent where $n = \underset{i \in [d]}{\max} \{n_i\}$ and the constant $C>0$ here depends on the sub-Gaussian norm of the distribution.

\end{theorem}

\begin{proof} Consider the map $p: \R^{r\sum_{i\in[d]} n_i} \rightarrow \R^{n_1 \times n_2 \dots \times n_d}$ where
$$ (\vx_{11}, \vx_{12}, \dots, \vx_{21}, \vx_{22} \dots \vx_{2r}, \dots \vx_{d1}, \vx_{d2} \dots \vx_{dr}) \mapsto \sum_{j=1}^{r} \vx_{1j} \out \vx_{2j} \out \dots \out \vx_{dj},$$
where $\vx_{ij} \in \R^{n_i}$ for $i\in[d]$ and $j\in[r]$. This is a polynomial map whose range is $\mathcal{S}_r^{CP}$, and whose coordinate functions have degree at most $d$. Then, the result follows from \cite[Theorem~3.12]{zhang2025covering}. 
\end{proof}
Then,  Theorem~\ref{thm:TIHT_hosvd} completes the story of why tensor iterative hard thresholding can be applied for recovery of low CP or HOSVD rank tensors from much fewer independent sub-Gaussian measurements.
\section{Face-splitting sketching and the benefits of trimming}

\label{sec:facespkit}

While sub-Gaussian measurement ensembles have nice theoretical guarantees, note that to recover $\tX$, we would need to store \revision{both} $\vb$ and the sub-Gaussian measurement matrix $\mat{A}$. Not only would storing $\mat{A} \in \R^{m \times n^d}$ be expensive, but matrix-vector operations  with this would have large computational overheads. 
This motivated the use of more database-friendly measurements that exploit the natural algebraic structure of the tensors and require storage memory that scales with $n$ rather than $n^d$. 

\subsection{Face-Splitting Measurements} 
One of the simplest representative examples of tensorial database-friendly measurements are face-splitting measurement ensembles. Given a collection of measurement matrices $\{\mat{A}_i \in \R^{m \times n_i}\}_{i \in [d]}$\revision{,} we define a measurement operator that acts on the space of tensors via the \emph{face-splitting product} of these matrices. The face-splitting product\revision{,} which we shall denote by the symbol $\bullet$\revision{,} is given by
$$\mat{A} := \mat{A}_1 \bullet \mat{A}_2 \dots \bullet \mat{A}_d =   \left(\mat{A}_1^T \odot \mat{A}_2^T \dots \odot \mat{A}_d^T\right)^T\in \R^{m \times n_1n_2 \dots n_d}.$$
where $\odot$ \revision{represents} the Khatri-Rao product.  Then, each row of \revised{$\mat{A}$} is a Kronecker product of the respective rows of the component matrices 
$$\mat{A} = \begin{bmatrix} 
\va_{11}^T \otimes \va_{21}^T \otimes \dots \otimes \va_{d1}^T  \\
\va_{12}^T \otimes \va_{22}^T \otimes \dots \otimes \va_{d2}^T  \\
 \vdots  \\
\va_{1m}^T \otimes \va_{2m}^T \otimes \dots \otimes \va_{dm}^T\\
\end{bmatrix},$$
\revised{where $\{\va_{ij}^T\}$ denotes the $j$-th row of the component matrix $\mat{A}_i$.}\revised{ Note that,} if the component matrices $\mat{A}_i$ are independent random matrices, then the rows of $\mat{A}$ are independent.

 As mentioned above, a key difficulty of using tensor-structured measurements is that such measurement operators lack non-trivial low-rank TensorRIP guarantees. So, establishing successful iterative methods recovering tensors from such measurements has been a topic of sustained interest, and has partial success in the matrix case $d = 2$ \cite{zhong2015efficient, foucart2019iterative,cai2015rop}. For comparison, Johnson-Lindenstrauss properties (RIP property over finite sets) are known to have the following form:

 \begin{theorem} [\protect{\cite[Theorem 2]{ahle2019almost}}]
 \label{thm:jl_face_splitting}
Given $\delta > 0, 0 <\revised{\psi} \leq e^{-2}$ and matrices  $\{\mat{A}_i\}_{i\in[d]}\subset \R^{m \times n}$ whose rows are sampled i.i.d. from a mean zero, isotropic sub-Gaussian distribution, let $$\mat{A} := \frac{1}{\sqrt{m}} \mat{A}_1 \, \bullet \, \mat{A}_2 \, \bullet \dots \bullet \, \mat{A}_{d}\;.$$Then, for any finite set $\mathcal{S} \subseteq \R^{n \times n \times \dots \times n}$,  $\mat{A}$ satisfies $(\delta, \mathcal{S})$-RIP with probability at least $1 - \revised{\psi}$ whenever
\begin{equation}\label{eq:kron-target-dim}
m \geq C^d\max \left\{ \delta^{-2} \ln \frac{|\mathcal{S}|}{\revised{\psi}} , \delta^{-1}\left(\ln\frac{|\mathcal{S}|}{\revised{\psi}}\right)^{d} \right\},
\end{equation}
for a constant $C \in \R^+$ that depends only on the sub-Gaussian norm of the rows of $\mat{A}$.  It is also shown that this bound is optimal in $\delta$ and $\revised{\psi}$ \cite[Theorem 4]{ahle2019almost}.
\end{theorem} 

\subsection{Failure of the RIP}
 A direct extension of this result to the set $\mathcal{S}$ of all low-rank tensors leads to unsatisfactory results. Moreover, it can be shown that uniform compression of the set of all low-rank tensors with any non-trivial ratio is impossible, as a consequence of such measurements essentially mimicking the low-rank structure of the tensors:

\begin{proposition} 
\label{prop:db_friendly_trip}
Consider a random measurement matrix $\mat{A} = \mat{A}_1 \bullet \mat{A}_2 \dots \bullet \mat{A}_d \in \R^{m \times n^d}$, where $\{\mat{A}_i\}_{i \in [d]}$ are random matrices whose rows are sampled from $\mathcal{D}$, a distribution of independent, mean-zero, sub-Gaussian random vectors with
sub-Gaussian norm bounded by $\ell$, whose coordinates are independent random variables that
are bounded away from zero, $|(\mat{A}_i)_{kl}| > q > 0$ a.s. for any $k \in [m]$ and $l \in [n]$.  Given any $\delta \in (0, 1)$ and $\phi<1/2$, we have that
$$ \lim_{ n \rightarrow \infty} \mathbbm{P}( \alpha_n \mat{A} \text{ satisfies } (\delta,r)\text{-TensorRIP} \text{ for some } m < n^{d(1 - \phi)} \text{ and } \alpha_n \in \R^+) = 0. $$

\end{proposition}
Here, TensorRIP fails for both low-CP and HOSVD vectors. Actually, we show the contradiction to the RIP on CP rank-$1$ tensors (or, equivalently, rank $(1,1, \ldots, 1)$-HOSVD tensors). Similar results were mentioned in the literature for matrices \cite{zhong2015efficient,cai2015rop} and tensors \cite{jiang2022near}. Yet, we believe that including a formal result and its proof is valuable for completeness of the discussion. The proof below relies on the approximate orthogonality of the rows inherited from the respective property of independent sub-Gaussian vectors in the form of the following
\begin{lemma}
\label{lem:orth_row}
Suppose that two random vectors, $\vu,\vv \in \R^n$\revision{,} are independent, mean-zero, sub-Gaussian random vectors with
sub-Gaussian norm bounded by $\ell$. Then, for any $ t> 0$, 
\begin{equation} \mathbbm{P}\left(|\langle \vu, \vv \rangle| \geq \ell^2\sqrt{n} t \right) \leq e^{-C \min\{t^2, \sqrt{n}t\}}, \end{equation} for some {absolute} constant $C \in \R^+.$ \end{lemma}
The proof of Lemma~\ref{lem:orth_row} follows along the proof of \cite[Lemma 6]{jeong2025linear} based on the Gaussian Chaos comparison lemma from \cite{hdp}.
\begin{proof}[Proof of \cref{prop:db_friendly_trip}]
Suppose that the matrix $\alpha_n\mat{A}$ satisfies $(\delta,r)$-TensorRIP for some $\delta < 1$. Let $\va_{1i}^\intercal \otimes \va_{2i}^\intercal \otimes \dots \otimes \va_{di}^\intercal$ denote the $i$-th row of the matrix $\mat{A}$. Let $\tX = \va_{11} \out \va_{21} \out \dots \out \va_{d1}$ and  $\tY = \vy_1 \out \vy_2 \dots \out \vy_d$, where $\{\vy_i\}$'s are random vectors sampled i.i.d from $\mathcal{D}$. Note that, $\algoname{Vec}({\tX}) = \va_{11} \otimes \va_{21} \otimes \dots \otimes \va_{d1} $ and $\algoname{Vec}({\tY})= \vy_1 \otimes \vy_2 \dots\otimes \vy_d $. Note that both $\tX$ and $\tY$ are rank $1$ tensors and their norms $\|\tX\|_F^2 , \|\tY\|^2_F > q^{2d}n^d$.

We have 
\begin{equation}
\label{eq:low-bound}
{1+\delta} \geq \alpha_n^2 \frac{\|\mathcal{A}(\tX)\|_2^2}{\|\tX\|_F^2} \geq \alpha_n^2 \|\va_{11} \otimes \va_{21} \otimes \dots \otimes\va_{d1}\|_2^2  >  \alpha_n^2 q^{2d} n^d \quad \text{a.s.}
 \end{equation}

At the same time
\begin{align}\label{eq:up-bound} {1-\delta} \leq \alpha_n^2\frac{\|\mathcal{A}(\tY)\|_2^2}{\|\tY\|_F^2} &< \frac{\alpha_n^2}{q^{2d}n^d} \sum_{i=1}^{m} |\langle \va_{1i} \otimes \va_{2i} \otimes \dots \otimes \va_{di},\vy_1 \otimes \vy_2 \otimes \dots \otimes \vy_d\rangle|^2 \nonumber\\
&\leq \alpha_n^2 m\underset{i\in[m],j\in[d]}{\max} \frac{|\langle \va_{ji},\vy_j \rangle|^{2d}}{q^{2d}n^d}. \end{align}

For any $\phi > 0$ and $\varepsilon > 0$ and any compressive $m \le n^d$, by Lemma~\ref{lem:orth_row} and the union bound,
\begin{equation}\label{eq:angle} \mathbbm{P} (\mathcal{E}) := \mathbbm{P}\left( \underset{i\in[m],j\in[d]}{\max} \frac{|\langle \va_{ji},\vy_j\rangle|^{2d}}{n^d} > \ell^{4d}{\varepsilon}^{2d} n^{\phi d} \right)  \leq n^{d} d \exp\left(-{ C\varepsilon^2} n^{\phi}\right).
\end{equation}
Then,
\begin{align*}
\mathbbm{P} &\left[1 - \delta \le \frac{\|\mathcal{A}(\tZ)\|_F^2}{\|\tZ\|_F^2} \le 1 + \delta \text{ for all rank-1 tensors $\tZ$}\right] \\
&\le \mathbbm{P} \left[1 - \delta \le \frac{\|\mathcal{A}(\tY)\|_F^2}{\|\tY\|_F^2} \text{ and } \frac{\|\mathcal{A}(\tX)\|_F^2}{\|\tX\|_F^2} \le 1 + \delta \text{ and } \mathcal{E}^c\right] + \mathbbm{P} (\mathcal{E}) \\
&\overset{\eqref{eq:low-bound},\eqref{eq:up-bound}}{\le} \mathbbm{P}\left[\frac{1+\delta}{1 - \delta} >  \frac{n^dq^{4d}}{ m\varepsilon^{2d} n^{\phi d}\ell^{4d}}\right] + n^{d} d \exp\left(-{ C\varepsilon^2} n^{\phi}\right).
\end{align*}
So, for any $\phi > 0$, if $m \le n^{d(1 - \phi)}$, we can set $\varepsilon = (\frac{q}{\ell})^2 \left(\frac{1-\delta}{1+\delta}\right)^{1/{2d}}$ to get 
\begin{align*}\mathbbm{P} &\left[1 - \delta \le \frac{\|\mathcal{A}(\tZ)\|_2^2}{\|\tZ\|_F^2} \le 1 + \delta \text{ for all rank-1 tensors $\tZ$}\right] \\
&\le 0 + n^d d \exp\left(- C (\frac{q}{\ell})^4 \left(\frac{1-\delta}{1 + \delta}\right)^{1/d} n^\phi\right) \xrightarrow{n \rightarrow \infty } 0,
\end{align*}
which concludes the proof of Proposition~\ref{prop:db_friendly_trip}.
\end{proof}

\begin{remark} While \cref{prop:db_friendly_trip} is stated for face-splitting product of matrices with independent sub-Gaussian entries bounded away from zero, the result also holds under weaker assumptions, e.g., relaxing almost sure separation from zero. For instance, we can replicate this proof to show that a similar result is true for the face-splitting product of matrices $\mat{A}_i \in \R^{m \times n}$ whose rows are sampled uniformly i.i.d. from $\revised{\sqrt{n}\mathbbm{S}^{n-1}}$. 
\end{remark}

\subsection{Uniform geometry preservation for trimmed face-splitting measurements}

In this section, we show that local trimming of the measurements can help establish stronger theoretical RIP-like guarantees.
\begin{definition}\label{def:trimmed-matrix}[Trimmed matrix $\mat{A}_k^\vx$]
    Consider a measurement matrix $\mat{A} \in \R^{m \times n}$ \revised{with rows $\{\va_i^T\}_{i \in [m]}$,} and any $\vx \in \R^n$ and $m > k > 0$. 
    Let $y_i = \va_i^T \vx$, and $|y_{i_1}| \leq |y_{i_2}| \leq \dots \le |y_{i_m}|$, let $i_1, i_2, \dots, i_{m-k}$ be all the first $m-k$ indices. 
Then, the $j$-th row of the trimmed matrix $\mat{A}_k^\vx \in \R^{ (m - k) \times n}$  is given by
$$ \sqrt\frac{{m}}{m-k} \va_{i_j} \quad j \in [m-k]\;.$$ 
\end{definition}
When the value of $k$ is clear, we suppress it and denote the trimmed matrix w.r.t $\vx$ as $\mat{A}^x$ instead. \revision{The main} result of this section is the following theorem. We formulate it for the component matrices of the same shape for the \revision{sake of simplicity in} exposition.
\begin{theorem}
\label{theorem:trim_trip}
Let $\mat{A} = \frac{1}{\sqrt{m}} \mat{A}_1 \, \bullet \, \mat{A}_2 \, \bullet \dots \bullet \, \mat{A}_{d},$ where $\mat{A}_i \in \R^{m \times n}$ for $i \in [d]$ are independent random matrices with i.i.d. mean zero, variance one, sub-Gaussian entries. \revised{Let $\delta, \eta \in (0,1)$ and $\phi \in \left(0, Cdnr/\log(n)\right]$ for some absolute constant $C > 0$.}
If 
\begin{align}
    &m \geq \rev{K\max\{C_1^{\rev{d}} \delta^{-2},1\}}\delta^{-2}(r^d + dnr) \ln\left(\frac{n^{d/2} (d+1)}{\delta}\right) \quad \text{ for } \mathcal{S} = \mathcal{S}_\vr^{HOSVD}, \text{ or }\label{eq:order_m_hosvd} \\
    &m \geq \rev{KC_3 \max\{C_1^{\rev{d}},2\}}\delta^{-2}dnr\ln\left(\frac{n^{3d/2}}{\delta}\right) \quad \text{ for } \mathcal{S} = \mathcal{S}_r^{CP}, \label{eq:order_m_cp}
    \end{align}
 then with probability $1 - cn^{-\phi}$, 
 $$ \sup_{\tX \in \mathcal{S}}
    \left  \vert \|\mat{A}_k^\vx\vx\|_2^2 - 1 \right \vert \leq \delta\;,$$
  where $\vx = \algoname{Vec} (\tX)$ and $\mat{A}^{\vx}_k$ denotes a trimmed matrix $\mat{A}$ as per Definition~\ref{def:trimmed-matrix} with
\rev{\begin{gather} 
k = \ln(8n^\phi) + (r^d + ndr) \ln\left(\frac{9(d+1)n^{d/2}}{\delta}\right) \text{ for } \mathcal{S} = \mathcal{S}_\vr^{HOSVD}, \text{ or }\\ 
k = \ln(8n^\phi) +   C_3 rnd\ln\left(\frac{3n^{3d/2}}{\delta}\right) \text{ for } \mathcal{S} = \mathcal{S}_\vr^{CP}
\end{gather}}
Here, $C_1, C_2, \rev{C_3}, c \in \R^{+}$, are constants that can depend on the sub-Gaussian norm of the entries of $\mat{A}_i$'s \rev{and $C, K \in \R^+$ are absolute constants}.
\end{theorem}

The proof of this theorem hinges on a known result about the trimmed mean estimator  for heavy-tailed distributions, which turn out to have sub-Gaussian properties.    Given  i.i.d. samples $\{x_i \in \R\}_{ i \in [n]}$ from a distribution, let $\hat{\mu}$ denote the sample mean and $\hat{\mu}_{k_1,k_2}$ denote the $(k_1,k_2)$-trimmed sample mean for some $k_1 + k_2 < n$. That is,
\begin{equation}
\label{eq:trim_mean} \hat{\mu} := \frac{1}{n} \sum_{i=1}^n x_i  \quad \text{ and } \quad \hat{\mu}_{k_1,k_2} := \frac{1}{n - k_1 - k_2} \sum_{i=k_1+1}^{n-k_2} x_{(i)},\end{equation}
where $x_{(1)} \leq x_{(2)} \dots \leq x_{(n)}$ denote the order statistics of the $\{x_i\}$. If $k_1 = 0$, we shall denote the trimmed mean estimator by $\hat{\mu}_{k_2}$ for brevity.

\begin{theorem}[\protect{\cite[Theorem 2.17]{rico2022optimal}}]
\label{thm: trim_sub-Gaussian}
Consider $p$ i.i.d random variables $x_1, x_2 \dots x_p$ in $\R$ that have been sampled from a non-negative distribution $\mathcal{D}$. Further, assume that the distribution has mean $\mu $ and variance $\sigma < \infty$. Then, there exist constants $C,c \in \R^+$ such that for any $\alpha \geq 4e^{-cp}$ we have
$$  \mathbbm{P}\left(| \hat{\mu}_{\revised{\ln(8/\alpha)}} - \mu | \leq C \sigma \sqrt{\frac{2\ln(4/\alpha)}{p}}\right) \geq 1 - \alpha$$
where \revised{$\hat\mu_{\ln(8/\alpha)} := \hat\mu_{k_1,k_2}$ is defined as in \eqref{eq:trim_mean} with $k_1 = 0$ and $k_2 = \ln(8/\alpha)$.}
\end{theorem}

This result applies to the face-splitting measurement matrices as follows.

\begin{lemma}
\label{lem:jl_trimming}
Let $\alpha, \delta \in (0,1)$ and $\mat{A} := \frac{1}{\sqrt{m}} \mat{A}_1 \, \bullet \, \mat{A}_2 \, \bullet \dots \bullet \, \mat{A}_{d},$ where all  $\mat{A}_i \in \R^{m \times n}$, $i \in [d]$ have i.i.d. mean zero, variance one, sub-Gaussian entries. 
Further, for any finite set $\{\tX_1, \tX_2 \ldots, \tX_p\} = \mathcal{S} \subseteq \R^{n \times n \times \dots \times n}$,  we consider their embedding with respect to the data-driven row trimming of $\mat{A}^{\vx_i}_k$ with $\vx_i = \algoname{Vec}(\tX_i)$ and $k = k = \ln(8\rev{p}/\alpha)$, 
 and $\alpha \geq 4e^{-cm}$. If $m \geq C^d \delta^{-2}\ln(4p/\alpha)$, then
$$\left|\left\|\mat{A}_{k}^{\vx_i} \vx_i\right\|_2^2 - \left\|\vx_i\right\|_2^2\right| \leq \delta  \left\|\vx_i\right\|_2^2\text{ for all } i \in [p]$$
with probability at least $1 - \alpha$.
Here, $C,\revised{c}\in \R^+$ that depends only on the sub-Gaussian norm of the entries of $\mat{A}_i$.
\end{lemma}

\revised{The proof of Lemma~\ref{lem:jl_trimming} relies on the following known bounds on the tails of the Kronecker product of independent random vectors.}
\revised{\begin{lemma}[\protect{\cite[Lemma 4.9]{ahle2020oblivious}}]
\label{lemma:GenKhintchine}
  Consider $\ell \geq 1$  and $k \in \mathbbm{Z}_{\geq 0}$. let  $\{\bm\sigma_{i} \in \R^{n_i}\}_{i \in [d]}$ be independent vectors that satisfy the Khintchine inequality $\| \langle \bm{\sigma}_i, \vx \rangle \|_\ell \leq C_\ell \|\vx\|_2 $ for any $\vx \in \R^{n_i}$. Then, for any vector $\vy \in \R^{n_1 n_2 \dots n_d}$,
  $$ \| \langle \bm\sigma_1 \otimes \bm\sigma_2 \otimes \dots \otimes \bm\sigma_d, \vy \rangle \|_\ell \leq C_\ell^d \|\vy\|^2.$$
\end{lemma}}

\revised{\begin{remark} 
\label{rem:khintchine}
Consider a random sub-Gaussian vector $\bm\sigma \in \R^n$ populated with i.i.d\revision{.} sub-Gaussian entries with sub-Gaussian norm $K$. Given any $\vx \in \R^n$, we have that $\|\langle \vx , \bm\sigma \rangle\|_{\psi_2} \leq C_1'K\|\vx\|_2$ for some absolute constant $C_1$. Then for any $\ell \geq 1$, by \cite[Proposition 2.6.6]{vershynin2018high}, $\|\langle \bm\sigma, \vx \rangle\|_\ell \leq C_2\sqrt{\ell}\|\vx\|_2$ for some constant $C_2$ that only depends on the sub-Gaussian norm $K$.
Then, for independent $\{\bm\sigma_{i} \in \R^{n_i}\}_{i \in [d]}$ consisting of i.i.d. sub-Gaussian entries and  $\ell \geq 1$, \cref{lemma:GenKhintchine} reduces to 
$$ \|\langle  \bm\sigma_1 \otimes \bm\sigma_2 \dots \otimes \bm\sigma_d, \vy \rangle\|_\ell\leq (C \ell)^{\frac{d}{2}}\|\vy\|_2,$$
for any $\vy$ in $\R^{n_1 n_2 \dots n_d}$ where $C$ is a constant that only depends on the sub-Gaussian norm $K$.
\end{remark}}

\begin{proof}[Proof of Lemma \ref{lem:jl_trimming}]
\revised{Let $\va_{kj}^T \in \R^n$ denotes the $j$th row of $\mat{A}_k$}. A direct check shows that
$$
  \mu := \E [  \langle \va_{1j} \otimes \va_{2j} \dots \otimes \va_{dj}, \vx \rangle^2 ] = \|\vx\|_2^2.
$$
for any $\vx \in \R^{n^d}$. 
Then, \revised{as a consequence of \rev{Khintchine} inequality (see \cref{rem:khintchine})}, for any $\ell \ge 1$,
\begin{equation}\label{eqn:genkhint_ineq1}
(\E|\langle  \va_{1j} \otimes \va_{2j} \dots \otimes \va_{dj}, \vx \rangle|^{\revised{\ell}})^{1/{\revised{\ell}}}\leq (C \ell)^{\frac{d}{2}}\|\vx\|_2
\end{equation}
and
$$
\sigma^2 := \E  [  \langle \va_{1j} \otimes \va_{2j} \dots \otimes \va_{dj}, \vx \rangle^4 ] - \left(\E [  \langle \va_{1j} \otimes \va_{2j} \dots \otimes \va_{dj}, \vx \rangle^2 ] \right)^2  \le \left[ (4C)^{2d} - 1 \right]\|\vx\|_2^4.
$$
So, applying Theorem~\ref{thm: trim_sub-Gaussian} to the i.i.d. random variables $x_i$ such that
$$ \left\| \mat{A} \vx \right\|_2^2 = \frac{1}{m} \sum_{ j = 1}^{m}  \langle \va_{1j} \otimes \va_{2j} \dots \otimes \va_{dj}, \vx \rangle^2 =: \frac{1}{m} \sum_{i=1}^m x_i,$$
we observe that $\hat{\mu}_{k_2} = \mat{A}_{k_2}^{\vx}$ and

$$  \mathbbm{P}\left(\frac{| \hat{\mu}_{k_2} - \|\vx\|_2^2 |}{\|\vx\|_2^2} \leq C' ((4C)^{2d} - 1)^{1/2} \sqrt{\frac{2\ln(4/\alpha)}{m}}\right) \geq 1 - \alpha $$
for any $\alpha \geq 4e^{-cm}$. So, for $|\|\mat{A}_{k}^{\vx}\vx\|_2^2 - \|\vx\|_2|^2 \leq \delta$ with probability at least $\alpha > 0$, we would need
$$ m \geq C^d \delta^{-2} \ln(4/\alpha),$$
for some new constant $C>0$. The desired result follows \revision{from} the union bound argument. 
\end{proof}

\begin{remark}
Note that \revision{by} utilizing trimming, we \revision{can} reduce the dependence of $m$ on $\ln \frac{|{\mathcal{S}|}}{\eta}$ from order $d$ in Theorem~\ref{thm:jl_face_splitting} to a linear dependence in Lemma~\ref{lem:jl_trimming}. This reduction helps us to produce TensorRIP-like guarantees.
\end{remark}

\begin{proposition}\protect{\cite[Lemma 3.1]{vershynin2020concentration}} \label{prop:probabilistic_norm_bound_khatri}
Let $\va = \va_1 \kron \va_2 \kron \ldots \kron \va_d$ be a vector in $\R^{n^d}$ such that $\va_i \in \R^n$ are independent random vectors populated with independent, mean zero, unit variance, sub-Gaussian coordinates. Then, for any $0 \le t \le 2n^{d/2}$, we have
$$
\mathbbm{P} \left( \|\va\|_2 > n^{d/2} + t\right) \le 2 \exp\left(- \frac{ct^2}{d n^{d-1}}\right), 
$$
where $c$ is a positive absolute constant that might only depend on the sub-Gaussian norms.

\end{proposition}

\begin{proof}[Proof of Theorem~\ref{theorem:trim_trip}]
    Given a set $\mathcal{S} \subseteq \mathbbm{S}^{n^d-1}$, let $\mathcal{N} \subseteq \mathbbm{S}^{n^d-1}$ represent an $\varepsilon$-net over it for some $\varepsilon$ which will be determined later. For any $\vx \in \mathcal{S}$, let $\hat{\vx} =  \hat{\vx}(\vx) \in \mathcal{N}$ be such that $\|\vx - \hat{\vx}\|_2 \leq \varepsilon$. Let $\mat{A}^{\vx}$ and $\mat{A}^{\hat{\vx}}$ denote the $k$-trimming of $\vx$ and $\hat{\vx}$ respectively. Without loss of generality, we can also assume that the rows $\{\va^\vx_i\}$ and $\{\va^{\hat{\vx}}_i\}$ of $\mat{A}^{\vx}$ and $\mat{A}^{\hat\vx}$ respectively have been reordered such that,
    $$ |\mat{a}^{\vx}_1{\vx}| \leq |\mat{a}^{\vx}_2{\vx}| \le \dots \le |\mat{a}^{\vx}_{m-k}{\vx}| \text{ and } |\mat{a}^{\hat{\vx}}_1{\hat{\vx}}| \leq 
    |\mat{a}^{\hat\vx}_2{\vx}|_2 \le \dots \le |\mat{a}^{\hat\vx}_{m-k}{\hat\vx}| .$$

    Note that
     \begin{align*}
        \sup_{\vx \in \mathcal{S}}\left|\|\mat{A}^\vx \vx\|^2_2 - 1\right| &=     \sup_{\vx \in \mathcal{S}}\left|\|\mat{A}^\vx \vx -\mat{A}^{\hat{\vx}} \hat\vx + \mat{A}^{\hat{\vx}} \hat\vx \|_2^2 - 1\right| \\
        &\leq  \sup_{\vx \in \mathcal{S}}\|\mat{A}^\vx \vx -\mat{A}^{\hat{\vx}} \hat\vx \|_2^2 + \sup_{\hat \vx \in \mathcal{N}}\left|\|\mat{A}^{\hat{\vx}} \hat\vx\|_2^2 - 1\right| + 2\sup_{\vx \in \mathcal{S}}\|\mat{A}^\vx \vx -\mat{A}^{\hat{\vx}} \hat\vx\|_2\|\mat{A}^{\hat{\vx}} \hat\vx\|_2. 
    \end{align*}
    Now, note that (e.g., \cite[Lemma 2]{zhang2018median}), $\|\vu_{(k)} - \vv_{(k)}\| \leq \|\vu-\vv\|_\infty$ for any two vectors $\vu$ and $\vv$ and their $k$-th order statistics. So, we have
$$ \|\mat{A}^\vx \vx -\mat{A}^{\hat{\vx}} \hat\vx\|_2^2  \leq m\|\mat{A} \vx -\mat{A}\hat{\vx}\|_{\infty}^2 \leq m\max_i\|\va_i\|_2^2 \varepsilon^2,$$
where $\va_i$ is the $i$-th row of $\mat{A}$. So, from Proposition~\ref{prop:probabilistic_norm_bound_khatri} with $t = n^{d/2}$,
$$
\mathbbm{P} (\mathcal{E}_1) := \mathbbm{P}\left( \text{exists } i \in [m], \quad \|\va_i\|_2^2 > \frac{2 n^d}{m}\right) \le 2m \exp\left(-\frac{2cn}{d}\right).
$$
Then, from \cref{lem:jl_trimming} \revised{with $\alpha = n^{-\phi}$}, we know that 
         $$
         \mathbbm{P} (\mathcal{E}_2) := \mathbbm{P}(\sup_{\vy \in \mathcal{N}}
    \left \vert\|\mat{A}^\vy \vy\|_2^2 - 1 \right \vert > \delta) \le 1/n^{\phi}$$
     as long as\begin{equation}\label{eq:m-bound}
m \ge \rev{\max\{C^d,2\}}\delta^{-2}\ln(4|\mathcal{N}|n^{\phi})\rev{\text{, and}}
    \end{equation}
\begin{equation}\rev{k = \ln(8|\mathcal{N}|n^\phi) = \ln(2) + \ln(4|\mathcal{N}|n^\phi) < \dfrac{m}{2} + 2}\end{equation} {for $\phi < C'\dfrac{dnr}{\log(n)}$ where $C,C' \in \R^+$} \rev{ are constants that can depend on the sub-Gaussian norm of the entries.} Putting everything together, we have 
$$
\mathbbm{P}(\sup_{\vx \in \mathcal{S}}\left|\|\mat{A}^\vx \vx\|^2_2 - 1\right| > 3\delta) \le \mathbbm{P}(\mathcal{E}_1) + \mathbbm{P}(\mathcal{E}_2) + \mathbbm{P}\left(m \frac{2 n^d}{m}\varepsilon^2 + 2 \sqrt{1 + \delta} \sqrt{\frac{2 n^d}{m}}\varepsilon > 2 \delta\right).
$$
Using that $\delta \le 1$, the last term is zero, e.g., if we set $\varepsilon = \delta n^{-d/2}/3$. Further, this choice of $\varepsilon$ makes \eqref{eq:m-bound} hold for both low HOSVD and CP tensors due to the net estimates from Remark~\ref{lem:covering_numbers_tensors} and our choices for $m$ as per \eqref{eq:order_m_hosvd} and \eqref{eq:order_m_cp}. Further, since the chosen $m$ scales linearly with $n$, we also get that 
$
\mathbbm{P} (\mathcal{E}_1)\le cn^{-\phi}$ for some constant $c>0$.  This concludes the proof of Theorem~\ref{theorem:trim_trip}.
\end{proof}

\section{Adaptive iterative hard thresholding recovery methods}\label{sec:iht-algorithms}
\subsection{Algorithms}
\label{section:proposed_methods}
Based on the properties of tensorial measurements as studied above, we propose two approaches to improve the iterative hard thresholding algorithm for low-rank tensor recovery.

\textbf{Trimmed Tensor Iterative Hard Thresholding (TrimTIHT).}   \revised{Our trimming-based hard thresholding approach is inspired by the works that successfully used trimming to restore sub-Gaussian-like properties in heavy-tailed data \cite{lugosi2019mean, oliveira2024improved, catoni2018dimension, chen2015solving}, as well as in robust recovery problems in the presence of heavy-tailed corruptions \cite{zhang2018median, li2020non, gorbunov2020stochastic}. } 

In \cref{alg:trimtiht}, 
we replace the gradient step of the TIHT algorithm with a scaled gradient estimate calculated based on a subset of rows of $\mat{A}$, chosen adaptively \revision{for} the current iterate to enhance its geometry preservation properties.
\RestyleAlgo{ruled}
\begin{algorithm}[!ht]
\DontPrintSemicolon
\caption{TrimTIHT}\label{alg:trimtiht}
\SetKwInOut{Input}{Input}
\SetKwInOut{Output}{Output}
\SetKwInOut{Initialize}{Initialize}
\Input{\emph{Measurement matrix} $\mat{A} \in \R^{m\times n_1  n_2 \dots  n_d}$ \emph{with rows} $\revised{\va_{i}^T}$ for $i \in [m]$; \emph{target rank} $\textbf{r} \in \mathbbm{N}^d$ for HOSVD or $r \in \mathbbm{N}$ for CP decomposition; \emph{measurements} $\vb \in \R^m$ where $\vb = \mat{A} \cdot \algoname{Vec}(\tX^\ast) + \bm\eta$ with noise $\eta \in \R^m$ and \emph{step sizes} $\gamma,\lambda > 0$; \emph{trimming parameter} $m_{trim} \in \mathbbm{N}$.}

\Initialize{$\tX^1 \gets \mat{0} \in \R^{n_1 \times n_2 \dots \times n_d}$}
\For{$k = 1, 2, \ldots, K$}
{ 
$\vy = (y_1, y_2, \dots y_m)^\intercal = | \mathcal{A}(\tX^k) - \vb |$

$[ i_1, i_2, \dots i_m]$ = \algoname{Sort}$(\vy)$ \tcp*{Returns indices of $\vy$ sorted in ascending order}

$\Theta_k = \left[i_1, i_{2}, \dots, i_{\rev{m - }m_{\text{trim}}} \right]$

  $\algoname{Vec}(\tY^{k+1}) \gets \algoname{Vec}(\tX^{k}) + \mu_k\frac{m}{|\Theta_k|} \sum\limits_{ j \in \Theta_k}  (b_j-\va_{j}^\intercal \algoname{Vec}(\tX^k)) \va_j$\;

    $\tX^{k+1} = T_{\mathbf{r}}(\tY^{k+1})$\; 
}
\Output{$\tX{}^{K+1}$, a rank $\vr$ tensor that estimates $\tX^\ast$}
\end{algorithm}

Consider the trimmed and rescaled measurements, namely, 
\begin{equation} \label{def:truncated-A} \mat{A}_k 
= \sqrt{\frac{m}{m-m_{\text{trim}}}} \mat{A}[\Theta_k,:], \qquad \mat{b}_k  
= \sqrt{\frac{m}{m-m_{\text{trim}}}} \vb[\Theta_k,:]. \end{equation}

Then, Step 5 of \cref{alg:trimtiht} can be expressed as
\begin{equation}\label{eq:trim-matversion} 
\algoname{Vec}(\tY^{t+1}) = \algoname{Vec}(\tX^{t}) + \mu_k \mat{A}_t^T(\vb_t - \mat{A}_t\algoname{Vec}(\tX^t)).
\end{equation}
This can alternatively be viewed as a gradient update with respect to the loss function $\frac{1}{2}\|\mat{A}_k \cdot - \vb_k\|^2$ starting at the point $\tX^k$ with step-size $\mu_k$. Note that, in the noiseless case $(\eta = 0)$, the trimmed matrix at iteration $k$, $\mat{A}_k$ corresponds to $\mat{A}_{m_{trim}}^{\algoname{Vec}(\tX^k - \tX^*)}$ as per Definition~\ref{def:trimmed-matrix}. It is important to note that we are able to perform trimming here as we have access to $\mathcal{A}(\tX^k - \tX^*) = \mathcal{A}(\tX^k) - \vb$ even though we do not have prior knowledge of $\tX^*$.  We provide convergence guarantees for Algorithm~\ref{alg:trimtiht} in Section~\ref{subsection:trimtiht}.

\textbf{Kaczmarz Tensor Iterative Hard Thresholding (KaczTIHT).} The second proposed method is a multi-order extension of the recently proposed \cite{zhang2015iterative, jeong2025linear} Kaczmarz Iterative Hard Thresholding method. Below, we show that \revision{its} linear convergence properties, as well as its practical advantage, persist in the tensor setting.   The details are presented in Algorithm \ref{alg:kziht}. 

\RestyleAlgo{ruled}
\begin{algorithm}[hbt!]
\caption{KaczTIHT}\label{alg:kziht}
\SetKwInOut{Input}{Input}
\SetKwInOut{Output}{Output}
\SetKwInOut{Initialize}{Initialize}
\Input{\emph{Measurement matrix} $\mat{A} \in \R^{m\times n_1  n_2 \dots  n_d}$ \emph{with rows} $\revised{\va_{i}^T}$ \revised{for $i \in [m]$}; \emph{target rank} $\textbf{r} \in \mathbbm{N}^d$ for HOSVD or $r \in \mathbbm{N}$ for CP decomposition; \emph{measurements} $\vb \in \R^m$ where $\vb = \mat{A} \cdot \algoname{Vec}(\tX^\ast) + \bm\eta$ with noise $\eta \in \R^m$ and ; \emph{step sizes} $\gamma,\lambda > 0$.

}
\Initialize{$\tX_1^1 \gets \mathbf{0} \in \R^{n_1 \times n_2 \dots \times n_d}$}

\For{$k = 1, 2, \dots, K$}{
$\tau \gets$ permutation of $[m]$\;
  \For{$j = 1, 2, \dots, m$}{
    $\algoname{Vec}({\tX_{j+1}^k} )\gets \algoname{Vec}(\tX_j^k) + \gamma \frac{\left(\vb_{\tau(j)}-\va_{\tau(j)}^\intercal \algoname{Vec}(\tX_j^k)\right) \va_{\tau(j)}}{\|\va_{\tau(j)}\|_2^2}$\tcp*{$m$ Kaczmarz steps}
  }
  $\tU^k = \tX_1^{k} + \lambda(\tX_{m+1}^k - \tX_1^k)$; \\
    $\tX_{1}^{k+1} = T_{\mathbf{r}}(\tU^k)$ \tcp*{Rank-$\vr$ hard thresholding step} 
}
\Output{$\tX{}_1^{K+1}$, a rank $\vr$ tensor that estimates $\tX^\ast$}
\end{algorithm}

While related techniques\revision{,} like Kaczmarz soft thresholding \cite{chen2021regularized} for tensor recovery\revision{,} employ the thresholding step after every Kaczmarz update, numerical evidence shows that this \revision{approach} is ineffective when a hard thresholding operator is used.  \revision{Furthermore}, numerical experiments also suggested that periodic thresholding every $\frac{m}{p}$ iterations, for some period $p$, as utilized in \cite[Algorithm 2]{jeong2025linear}\revision{,} did not help improve convergence. Also, random sampling with replacement offers similar, yet slightly slower\revision{,} rates of convergence when compared to reshuffling. 

Another related method is the stochastic iterative hard thresholding algorithm for low-rank HOSVD recovery 
(\cite{grotheer2020stochastic}), where the full gradient step is replaced by a block stochastic gradient descent step followed by hard thresholding. However, in StoTIHT \cite{grotheer2020stochastic}, the projection blocks are themselves large so that each of \revision{them} satisfies RIP, which is crucially different from our setting. \revised{ We provide convergence guarantees for Algorithm~\ref{alg:kziht} in Section~\ref{subsection:kacztiht}.}

\revised{\subsection{Computational costs} 
Let us briefly compare the computational complexity of the proposed methods \revision{with that of} the baseline TIHT (defined by \eqref{eq:iht-grad},\eqref{eq:iht-fit}). In all three methods, the cost of the thresholding step remains the same as in \eqref{eq:iht-fit}. Hence, we restrict our comparison to the cost \revision{of} the gradient update. The gradient update step in TIHT requires $O(m\prod_in_i)$ flops. For TrimTIHT (\cref{alg:trimtiht}), the cost of computing $\tY_k$ can be broken down into the cost of sorting to determine the row indices to keep, and a gradient update step resulting in $O(m\prod_i n_i + m\log(m))$ flops. For KaczTIHT, each Kaczmarz step requires $O(\prod_in_i)$ flops. To compute the estimate $\tU_k$ in \cref{alg:kziht} we use $m$ such Kaczmarz steps\revision{,} bringing the total to $O(m\prod_i n_i)$ flops. Given that $m \ll \prod_i n_i$, the three costs are compatible. We note that\revision{,} due to the efficiency of matrix computations, empirically on large-scale data, TrimTIHT\revision{, as} realized via equation \eqref{eq:trim-matversion}, performs faster per iteration than KaczTIHT.}

\subsection{Analysis of TrimTIHT}
\label{subsection:trimtiht}
First, we derive convergence guarantees for the TrimTIHT algorithm. Theorem~\ref{thm:main_path} below provides \revision{a} deterministic analysis of the tensor iterative hard thresholding algorithm, \revision{which} can be applied to the full measurement matrix $\mat{A}$ as well as to its trimmed\revision{,} iteration-adapted version $\mat{A}_t$. In Remark~\ref{rem:trimtiht-discussion}, we discuss the result and explain why it suggests that we can improve convergence by controlling the embedding norm distortion of tensors locally. In turn, Theorem~\ref{theorem:trim_trip} above guarantees that in the interesting case of face-splitting measurements, local trimming does successfully control local distortion of the norm. 

\begin{theorem}[Analysis of  TrimTIHT] \label{thm:main_path} Consider a tensor $\tX^\ast \in \R^{n \times n \dots \times n}$ of HOSVD rank $\vr = (r,r, \dots,r)$ (or CP rank $r$). Let $\tR^t = \tX^t - \tX^\ast$ be the residuals of the TrimTIHT \cref{alg:trimtiht} applied to the linear measurements $ \vb = \mathcal{A}(\tX^\ast)$. Recall that the dynamic of \cref{alg:trimtiht} is defined by a trimmed sub-matrix $\mat{A}_t$ of the matricization $\mat{A}  \in \R^{m \times n^d}$ of the measurement operator $\mathcal{A}$, defined as per \eqref{def:truncated-A}. 
Then, 
\begin{gather}
    \label{eqn:conv_path}
\|\tR^{t+1}\|_F \leq \left( 2\sqrt{ 1 - (2 - \mu_t\rho_t )\mu_t\Delta_t}+\sqrt{ 2\xi_t + \xi_t^2}\left(1 + \mu_t \|\mat{A}_t\|_{2} \sqrt{\Delta_t}\right) \right)^{t+1} \|\tR^{\rev{t}}\|_F,
\end{gather}
where
\begin{align}
\label{eqn:delta}
&\Delta_t := \frac{\|\mat{A}_t\vr^t\|_2^2}{\|\vr^t\|_2^2} , \text{ where } \vr^t := \algoname{Vec}(\tR^t) \text{ for any } t \ge 1\\
\label{eqn:rho}
&\rho_t := \frac{\|\mat{A}_t \vu_t \|_2^2}{\|\vu_t\|_2^2}, \text{ where } \vu_t := P_{\Omega_t} (\mat{A}_t^T\mat{A}_t\vr^t) \text{ and } \Omega_t := span\{\vr^t, \vr^{t+1}\},\\
\text{ and } &\xi_t \text{ is such that } \|\tX^{t+1} - \tY^{t} \|_F \leq (1 + \xi_t) \| \tY^t -\tX^\ast\|_F. 
\end{align}
Here, $P_{\Omega_t}$ denotes the orthogonal projection operator onto a linear subspace $\Omega_t$.
\end{theorem}

\begin{remark}\label{rem:trimtiht-discussion} Before proceeding to the proof of \cref{thm:main_path}, let \revised{us} discuss the result:
\begin{itemize}
    \item We focus on a version of \revision{the} exact low-rank recovery problem for brevity of the result. It is possible to obtain the noisy low-rank recovery version of Theorem~\ref{thm:main_path} via a similar analysis. 
    \item \textbf{On $\xi_t$:} From  \cref{thm:opt_hosvd_approx} and \cref{thm:opt_cp_approx}, we know that one can make $\xi_t$ arbitrarily small in the HOSVD-rank case, and in the CP-rank case under some additional boundedness assumptions. However, as previously noted, the algorithms we do use in practice for low-rank approximation might not have as good a guarantee. In the above theorem, we provide a flexible analysis \revision{that allows} the convergence rate \revision{to be} based on the level of low-rank approximation error at each iteration. 
    
    Moreover, if we assume that, in \revision{the} spirit of a known assumption for TIHT convergence (\cite{rauhut2017low}, see Theorem~\ref{thm:TIHT_hosvd} above),
$$\xi_t \leq \frac{16\min\{|1- \mu_t\Delta_t|^2\}}{17\left(1 + \|\mat{A}_t\|_{2} \mu_t \sqrt{\Delta_t}\right)^2},$$
 \eqref{eqn:conv_path} simplifies to
\begin{equation}
\label{eqn:conv_path_err}\|\tR^{t+1}\|_F
\le \left(2\sqrt{ 1 - 2\mu_t\Delta_t + \mu_t^2\rho_t \Delta_t  }+|1 - \mu_t\Delta_t|\right)\|\tR^t\|_F =: \alpha_t\|\tR^t\|_F .
\end{equation}
\item \textbf{On $\Delta_t, \rho_t$ and the effect of trimming:} From \eqref{eqn:conv_path} or  \eqref{eqn:conv_path_err} we can see that good convergence is ensured by $\mu_t \Delta_t \sim 1$ and $\mu_t \rho_t \sim 1$ for every $t$. As we do not have direct access to $\Delta_t$ (or $\rho_t$), choosing \revision{a} suitable iterate-dependent $\mu_t$ is nontrivial. On the other hand, focusing on the \emph{constant} step size $\mu$\footnote{Since one can rescale the operator $\mathcal{A}$, let \revised{us} assume a trivial step size $\mu = 1$.}, for a suitably chosen parameter $m = m(\delta)$ and $m_{\text{trim}}$, it is possible to theoretically ensure that $|1- \mu {\Delta_t}| \leq \delta$ for some  $0 <\delta < 1$ with high probability.

 \textbf{Connection to TensorRIP:} For example, it is straightforward to check that if the linear measurement operator $\mathcal{A}: \R^{n \times n \dots n} \rightarrow \R^{m}$ satisfies $(\delta,3\vr)$-TensorRIP, 
 $m_{trim} = 0$, then $|1 - \rho_t|, |1 - \Delta_t| < \delta$ and Theorem~\ref{thm:main_path} (in the simplified form \eqref{eqn:conv_path_err}) recovers the convergence guarantees for TIHT as long as $\delta < 0.106$.

\textbf{Implications for face-splitting measurements:} Moreover, in Theorem~\ref{thm:main_path}, $\Delta_t$ and $\rho_t$ can also be iterate-dependent, which  allows \revision{us} to incorporate the analysis of adaptive trimming. For instance, if $\mat{A}$ is given by the face-splitting product of bounded measurements (that do not satisfy TensorRIP as shown in Proposition~\ref{prop:db_friendly_trip}), \cref{theorem:trim_trip} applied to $\tY = \tX^t - \tX^*$ with a sufficiently large $m$ and $m_{trim}$ \rev{chosen to be a suitable fraction of $m$} ensures that with high probability $| 1 - \Delta_t| < \delta$ for any $t$.

Note that $\rho_t$ is more challenging to control. Remark~\ref{rem:rho-delta-comp} below checks that we always have $\rho_t \ge \Delta_t$, it also  depends on ``the future residual" $\tR^{t+1}$, which in turn is determined by the thresholding approximation at iteration $t$ and is not directly controlled by the iteration-dependent trimming. Only in special cases, like when $\tR^{t+1}$ and $\tR^{t}$ are highly coherent, our theory yields that trimming would help suitably control the value of $\rho_t$ and consequently $\rho_t/\Delta_t$. 
Nonetheless, experimental evidence (see \cref{fig:trim_delta} below) suggests that trimming the rows helps us control $\rho_t$ as well, aiding convergence.

\begin{figure}[!ht]
\includegraphics[width =\linewidth]{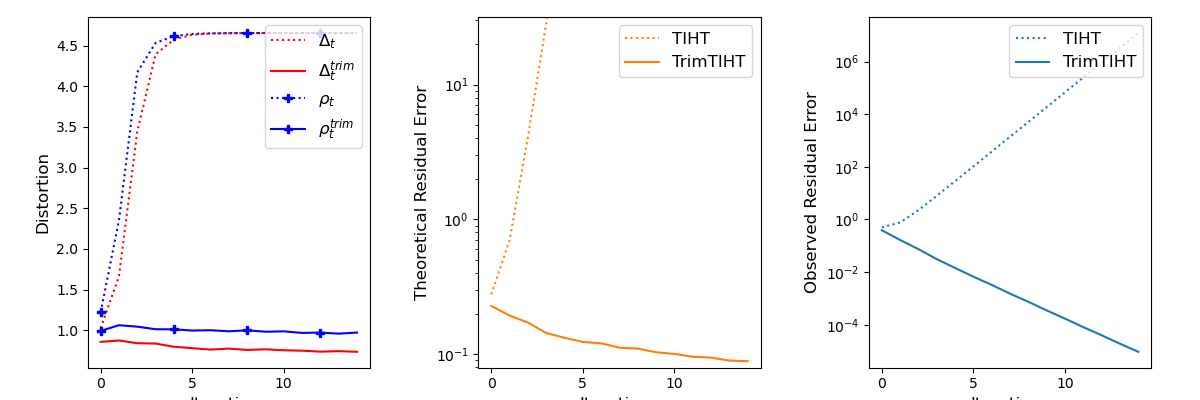}
\caption{Effect of trimming in the recovery of HOSVD rank-(2,2,2) tensor in $\R^{20 \times 20 \times 20}$ compressed to $m = 4250$ via the face-splitting product of sub-Gaussian measurements. As we can see from the plots, the observed residual error rate is better \revision{than} what the theory suggests.}
\label{fig:trim_delta}
\end{figure}
 
\item \textbf{On $\mu_t$: effect of step-size.} While the focus of this study is to understand the effect of trimming, the flexibility of the TrimTIHT analysis in \cref{thm:main_path} provides insight into why non-constant step-size $\mu_t$ can help even when TensorRIP guarantees fail, see Figure~\ref{fig:overlay_feas_vals}. The iteration dependent step-sizes are \revision{widely used} in practice, but the \revision{existing} theoretical analysis of such methods, including NTIHT~(Normalised Tensor Iterative Hard Thresholding) \cite{rauhut2017low}\revision{,} usually comes with strong TensorRIP requirements. 
\begin{figure}[ht]
\includegraphics[width=0.3
\linewidth]{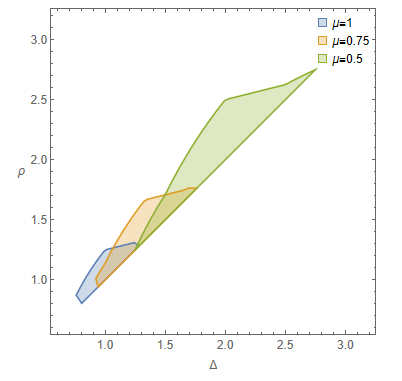}
\caption{Values of $\rho_t$ and $\Delta_t$, such that $ \alpha_t < 1$ for $\mu = 1, 0.75$ and $0.5$. For smaller step sizes, the range of acceptable parameters is \revision{larger} and also shifts to allow \revision{for} bigger distortion.} 
\label{fig:overlay_feas_vals}
\end{figure}
\end{itemize}
\end{remark}

Now we are ready to proceed with the proof of \cref{thm:main_path}, starting with the following key helper lemmas. 

\begin{lemma}
\label{lemma:strong_convexity}
 Let all quantities be as defined in the statement of \cref{thm:main_path}, in particular, $\mat{A} \vx^* =\vb$. For some $\vx$ and $\vy \in \R^{n^d}$, let $\rho > 0$ be such that
\begin{equation}\label{eq:lem3-cond} \|\mat{A}_t(\vx-\vy)\|_2^2 \leq \rho \|\vx-\vy\|_2^2.
\end{equation}
Then, for a function $f(\vz) := \|\mat{A}_t\vz - \vb_{\revised{t}}\|_2^2/2$, it holds
\begin{equation}
\label{eq:restricted_convexiry}
    f(\vy) \leq f(\vx) + \left\langle \mat{A}_t^T\mat{A}_t\left(\vx - \vx^* \right) , \vy-\vx\right\rangle +   \frac{\rho}{2}\|\vy-\vx \|_2^2. 
\end{equation}
\end{lemma}

\begin{proof}
 Let $\chi(s) := \vx + s(\vy-\vx) $ for $s \in [0,1]$.  From the fundamental theorem of calculus, using that $\nabla f(\vz) = \mat{A}_t^T\mat{A}_t(\vz - \vx^*)$, we have
    \begin{align*} f(\vy) - f(\vx) = \int_0^1 (f\circ\chi)'(s) ds &= \int_0^1 \nabla f ( \chi(s)) \chi'(s) ds\\
    &= \int_0^{1} \left\langle \mat{A}_t^T\mat{A}_t\left(\vx + s(\vy-\vx) - \vx^* \right) , \vy-\vx\right\rangle ds\\
&= \left\langle \mat{A}_t^T\mat{A}_t\left(\vx - \vx^* \right) , \vy-\vx\right\rangle  +  \int_0^{1} \left\langle s \mat{A}_t^T\mat{A}_t\left(\vy-\vx \right) , \vy-\vx\right\rangle ds\\
&= \left\langle \mat{A}_t^T\mat{A}_t\left(\vx - \vx^* \right) , \vy-\vx\right\rangle +   \frac{1}{2}\|\mat{A}_t(\vy-\vx) \|_2^2 \\
& \overset{\eqref{eq:lem3-cond}}{\le} \left\langle \mat{A}_t^T\mat{A}_t\left(\vx - \vx^* \right) , \vy-\vx\right\rangle +   \frac{\rho}{2}\|\vy-\vx \|_2^2.
    \end{align*}
    This immediately gives \eqref{eq:restricted_convexiry}.
\end{proof}

\begin{lemma}
\label{lem:rho_dist}
   Let all quantities be as defined in the statement of \cref{thm:main_path}. 
    Then, 
   $$ \|P_{\Omega_t} (\mat{A}_t^T\mat{A}_t\vr^t) \|_F^2 \leq \rho_t\|\mat{A}_t\vr^t\|_F^2 \quad \text{ where } \vr^t := \algoname{Vec}(\tR^t).
   $$
\end{lemma}
\begin{proof}
The definition of $\rho_t$ \eqref{eqn:rho} implies that  \eqref{eq:lem3-cond} holds with $$\vx = \algoname{Vec}(\tX^t) \text{ and } \vy = \vx - \frac{1}{\rho_t} P_{\Omega_t}(\mat{A}_t^T\mat{A}_t\vr^t).$$ We can apply Lemma~\ref{lemma:strong_convexity} as follows: 
\begin{align*}
f\big(\vx-& \frac{1}{\rho_t} P_{\Omega_t} (\mat{A}_t^T\mat{A}_t\vr^t)\big) \\
 &\leq f(\vx) - \left\langle \mat{A}_t^T\mat{A}_t\vr^t, \frac{1}{\rho_t} P_{\Omega_t} (\mat{A}_t^T\mat{A}_t\vr^t) \right \rangle + \frac{\rho_t}{2} \left\|\frac{1}{\rho_t}P_{\Omega_t} (\mat{A}_t^T\mat{A}_t\vr^t)\right\|_2^2\\
&\leq f(\vx) - \left\langle P_{\Omega_t}  (\mat{A}_t^T\mat{A}_t\vr^t), \frac{1}{\rho_t} P_{\Omega_t}  (\mat{A}_t^T\mat{A}_t\vr^t) \right \rangle + \frac{1}{2\rho_t} \left\|P_{\Omega_t} (\mat{A}_t^T\mat{A}_t\vr^t)\right\|_2^2\\
  &\leq \frac{1}{2} \| \mat{A}_t\vr^t\|_2^2 - \frac{1}{2\rho_t} \left\|P_\Omega (\mat{A}_t^T\mat{A}_t\vr^t)\right\|_2^2,
  \end{align*}
  where in the last step we observe that $ f(\algoname{Vec}(\tX^t)) = \frac{1}{2} \|\mat{A}_t\tX^t - \mat{A}_t\tX^*\|_2^2 = \frac{1}{2} \| \mat{A}_t\algoname{Vec}(\tR^t)\|^2$.
  Now, since $f(.) \ge 0$ we can conclude from above 
  $$
  \|P_{\Omega_t} (\mat{A}_t^T\mat{A}_t\algoname{Vec}(\tR^t)) \|_2^2 \leq \rho_t\|\mat{A}_t\algoname{Vec}(\tR^t)\|_2^2. 
  $$
\end{proof}

\begin{proof}[Proof of \cref{thm:main_path}]
From the definition of $\xi_t$, we can conclude that for $\tR^{t+1} = \tX^{t+1} - \tX^*$, 
$$
 \| \tR^{t+1} + \tX^* - \tY^t \|_F^2 \leq (1 + 2\xi_t + \xi_t^2) \| \tY^t - \tX^\ast \|_F^2,
 $$
 and so
$$\| \tR^{t+1}  \|_F^2 \leq 2 \underbrace{\langle \tR^{t+1}, \tY^t - \tX^* \rangle}_{\text{\tt{Term 1}}}  + ( 2\xi_t + \xi_t^2) \underbrace{\| \tY^t - \tX^\ast \|_F^2}_{\text{\tt{Term 2}}}.
$$

\textbf{Bounding \tt{Term 1}:} \text{Let }$\Omega_t' = span\{\tR^{t+1}\} \subseteq span\{\algoname{Vec}(\tR^t),\algoname{Vec}(\tR^{t+1})\} = \Omega_t$ and recall that $ \algoname{Vec}(\tY^t - \tX^\ast) =  (\mat{I} - \mu_t \mat{A}_t^T\mat{A}_t) \vr^t$, where $\vr^t := \algoname{Vec}(\tR^t)$. We can bound
\begin{align*}
\langle \tR^{t+1} , \tY^t - \tX^* \rangle  &= \|\tR^{t+1}\|_F \|P_{\Omega'_t}\left((\mat{I} - \mu_t\mat{A}_t^T\mat{A}_t)\vr^t\right)\|_2 \\
&\overset{\Omega'_{t} \subseteq \Omega_t}{\leq} \|\tR^{t+1}\|_F\|P_{\Omega_t}\left((\mat{I} - \mu_t\mat{A}_t^T\mat{A}_t)\vr^t\right)\|_2\\
&\overset{(a)}{\leq}\|\tR^{t+1}\|_F \sqrt{\|\vr^t\|_2^2  + \rho_t \mu_t^2 \|\mathcal{A}_t(\tR^t)\|_2^2 - 2 \mu_t \langle \vr^t, P_{\Omega_t}(\mat{A}_t^T\mat{A}_t\vr^t
) \rangle} \\
&\leq \|\tR^{t+1}\|_F\sqrt{\|\vr^t\|_2^2  + \rho_t \mu_t^2\|\mat{A}_t\vr^t\|_2^2 -2\mu_t\|\mat{A}_t\vr^t\|_2^2}   \\
&\overset{(b)}{\leq}\|\tR^{t+1}\|_F \sqrt{\|\tR^t\|_F^2 \left(  1  - (2 -\mu_t\rho_t )\mu_t\Delta_t \right)}.
\end{align*}

Here, inequality $(a)$ follows from  Lemma~\ref{lem:rho_dist} and $(b)$ follows from the definition of $\Delta_t$ in \eqref{eqn:delta}. 

\textbf{Bounding \tt{Term 2}:}
\begin{align*}
    \| \algoname{Vec}(\tY^t - \tX^\ast) \|_2  &\leq \|\tR^t\|_F + \mu_t \|\mat{A}_t^T\mat{A}_t \revised{\vr^t}\|_2\\
    &\leq \|\tR^t\|_F + \mu_t\|\mat{A}_t\|_2\|\mat{A}_t \revised{\vr^t}\|_2\\
    &= \left(1 + \mu_t\|\mat{A}_t\|_{2} \sqrt{\Delta_t}\right)\|\tR^t\|_F\;.
\end{align*}

Combining it all, we get that
$$ \|\tR^{t+1}\|_F^2 \leq 2\sqrt{  1 - (2 - \mu_t\rho_t )\mu_t\Delta_t }\|\tR^{t+1}\|_F\|\tR^t\|_F +  \revised{(2\xi_t + \xi_t^2)}\left(1 + \|\mat{A}_t\|_{2} \mu_t \sqrt{\Delta_t}\right)^2\|\tR^t\|_F^2.$$
This implies that there exists $0 \leq \alpha \leq 1 $ such that  
\begin{align*} 
&(1- \alpha) \|\tR^{t+1}\|_F^2 \leq 2\sqrt{ 1 - (2 - \mu_t\rho_t )\mu_t\Delta_t  }\|\tR^{t+1}\|_{\revised{F}}\|\tR^t\|_F\;.\\
&\alpha \|\tR^{t+1}\|_F^2 \leq ( 2\xi_t + \xi_t^2)\left(1 + \|\mat{A}_t\|_{2} \mu_t\sqrt{\Delta_t}\right)^2\|\tR^t\|_F^2\;. 
\end{align*}
Combining these gives us,
\begin{align*}
\left( 1 - \alpha + \sqrt{\alpha} \right) \|\tR^{t+1}\|_F &\leq \left( 2\sqrt{  1 - (2 - \mu_t\rho_t )\mu_t\Delta_t  }+ \sqrt{2\xi_t + \xi_t^2}\left(1 + \mu_t\|\mat{A}_t\|_{2} \sqrt{1 + \delta_t}\right) \right) \|\tR^t\|_F\;.
\end{align*}
However, the function $(1-\alpha + \sqrt{\alpha} )\geq 1$ on the interval $[0,1]$. So,
\begin{align*}
\|\tR^{t+1}\|_F 
\leq \left( 2\sqrt{ 1 - (2 - \mu_t\rho_t )\mu_t\Delta_t  }+\sqrt{ 2\xi_t + \xi_t^2}\left(1 + \mu_t \|\mat{A}_t\|_{2} \sqrt{\Delta_t}\right) \right) \|\tR^t\|_F,
\end{align*}
which concludes the proof of Theorem~\ref{thm:main_path}.
\end{proof}

\begin{remark} \label{rem:rho-delta-comp} Lemma~\ref{lem:rho_dist} also implies that $\rho_t \geq \Delta_t$ for any $t \ge 1$. Indeed, note that the norm of the projection of $\mat{A}_t^T\mat{A}_t\vr^t$ onto $\Omega = span\{\vr^t,\vr^{t+1}\}$ is not smaller than the norm of its projection onto the subspace $\Omega' = span\{\vr^t\}$. Thus,
\begin{align*} \rho_t \|\mat{A}_t\vr^t\|_2^2 &\geq \|P_\Omega(\mat{A}_t^T\mat{A}_t\vr^t)\|_2^2 \geq \|P_{\Omega'}(\mat{A}_t^T\mat{A}_t\vr^t)\|_2^2 \\ &\geq \left\| \langle \vr^t, \mat{A}_t^T\mat{A}_t\vr^t \rangle \frac{ \vr^t}{\|\vr^t\|_2^2}\right\|_2^2\geq\frac{\|\mat{A}_t\vr^t\|_2^4}{\|\vr^t\|_2^2} \revised{= \Delta_t \|\mat{A}_t\vr^t\|_2^2}\;.
\end{align*}
This gives us \revised{{$\rho_t \geq \Delta_t$}}.

\end{remark}

\subsection{Analysis of KaczTIHT}
\label{subsection:kacztiht}
Next, we give the convergence result for KaczTIHT. In Theorem~\ref{thm:main} below, we show that the algorithm \revision{exhibits} an exponential convergence rate given good Gaussian-like measurements and \revision{an} efficient tensor low-rank fitting subroutine. We formalize these assumptions as follows:

\textbf{\emph{Assumption 1}: (Distribution)} Let the matrix $\mat{A} \in \R^{m \times n^d}$ of a linear measurement operator $\mathcal{A}: \R^{n \times n \dots \times n} \rightarrow \R^m$ satisfy the following properties: 
its rows $\{\va_i^T\}_{i \in [m]} \subset \revised{\sqrt{n^d}\mathbbm{S}^{n^d-1}}$ are random vectors sampled independently from a mean zero, isotropic sub-Gaussian distribution with a sub-Gaussian norm $\ell$. 

\medskip

\revised{The assumption that all the modes of the tensor have the same dimension $n$ is made for simplicity of exposition. Further, we can always properly rescale the rows of $\mat{A}$. However,} the sub-Gaussian assumptions are quite strong. As a result, our subsequent theoretical guarantees \revised{hold for} classes of standard RIP-measurements, such as rescaled Gaussian measurements. \revised{Also}, even stronger results can be similarly concluded for bounded orthogonal measurement ensembles, such as those based on randomly subsampled \revision{discrete} Fourier or Hadamard transforms --- similarly to the vector case in \cite{jeong2025linear} -- see Remark~\ref{rem:bos} below. However, explaining the empirically observed advantage of KaczTIHT over TIHT, especially on non-RIP measurements such as tensor-structured measurements as considered above, should be a very interesting avenue for future work. 

Again, a good low-rank fitting routine is essential for the successful thresholding step:

\textbf{\emph{Assumption 2}: } Assume that the thresholding operation $T_\vr$ used in \cref{alg:kziht} is such that
\begin{equation}
\label{eq:extra_assmpn_4}
\| \tU^k- \tX_1^{k+1} \|^2_F \leq (1+\xi)^2 \| \tU^k - \tX^{\ast}\|_F^2 \;,\\
\end{equation}
for
\begin{equation*}
\xi = \frac{\delta^2}{ 5\left(1 + \delta + \left\|\frac{1}{\sqrt{m}}\mat{A}\right\|_{2} \sqrt{1 + \delta}\right)^2}\;.
\end{equation*} Note that, if we use the algorithm described in Theorem~\ref{thm:opt_hosvd_approx} with $\varepsilon = \xi$, then \eqref{eq:extra_assmpn_4} automatically holds for the thresholding sub-routine in KaczTIHT \ref{alg:kziht}. The same is true for the low rank CP thresholding using Theorem~\ref{thm:opt_cp_approx} for a suitably chosen $\varepsilon$ under mild assumptions.

\begin{theorem}[Analysis of  KaczTIHT]  \label{thm:main} Consider a tensor $\tX^\ast \in \R^{n \times n \dots \times n}$ of HOSVD rank $\vr = (r,r, \dots,r)$ (or CP rank $r$). Let $\tR^t = \tX_1^t - \tX^\ast$ be the residuals of the KaczTIHT \cref{alg:kziht} applied to the linear measurements $ \vb = \mathcal{A}(\tX^\ast) + \boldsymbol{\eta}$, where $\revised{\boldsymbol{\eta} \in \R^m}$ is the measurement noise and $\mathcal{A}: \R^{n \times n \dots \times n} \rightarrow \R^m$ is a linear measurement operator satisfying Assumption 1. Assume that for some $\delta \in \left(0, \frac{1}{5}\right)$, the step-sizes $\gamma, \lambda >0$ in \cref{alg:kziht} are taken to be
\begin{gather}
\label{eq:parameters}
    \gamma \leq \frac{\delta}{2\sqrt{n^d} m \ell^2 \ln^2m}, \quad \lambda = \frac{n^d}{m\gamma}\;,
\end{gather}
and Assumption~2 on the thresholding operator holds. 

Then, as long as 
\begin{gather*}
        m \geq C' \delta^{-2} ((3r)^d + 9nrd)\ln(d)\ln^2(m) \ \text{for HOSVD rank $\vr$  or }\\
        m \geq C'  \delta^{-2} ((3r)^d + 3nrd)\ln(d\revised{nr})\ln^2(m) \ \text{for CP rank-$r$, }
        \end{gather*}
with probability $1 - (m^2+1)e^{-C \ln^2m}$, we have
\[ \|\tR^{k+1}\|_F \leq (5\delta)^{k+1} \|\tR^0\|_F + \frac{1}{m}\frac{\delta + 2}{1-5\delta}\left(\left( \frac{\delta}{2(\ell\ln m)^2} + \frac{\delta^2}{2n^d(\ell\ln m )^2}\right)\|\bm{\eta}\|_1 + 
    \|\mat{A}^\intercal \bm{\eta}\|_2 \right).\]
Here, $C, C'>0$ are absolute constants that depend on the sub-Gaussian norm $l$, $\mat{A}$ is the \revised{linear matrix corresponding to} the measurement operator $\mathcal{A}$.
\end{theorem}

Before proceeding to the proof of \cref{thm:main}, let \revised{us} take a closer look at the updates performed by \revision{the} inner relaxed Kaczmarz loop in the KaczTIHT \cref{alg:kziht} algorithm.\revised{ For this result, let $\va_i$ for $i = 1, \ldots, m$ to denotes the conjugate transpose of the rows of $\mat{A}$, and $\tau$ is the permutation of the rows chosen for the $k$-th epoch. Then, we can define $\mat{B}_{\tau(k), \mat{B}_\tau}$, and $\bm{\eta}_{err}$ as follows} --
\begin{align}\label{eq:b-def}
     \mat{B}_{\tau(k)} :=  \frac{\gamma^2}{n^{2d}} &\sum_{m - k + 1 \leq i_1 < i_2 \leq m} \va_{\tau(i_2)}\va_{\tau(i_2)}^\intercal \va_{\tau(i_1)}\va_{\tau(i_1)}^\intercal \\ &- \frac{\gamma^3}{n^{3d}}\sum_{m-k+1 \leq i_1 < i_2 < i_3 \leq m} \va_{\tau(i_3)}\va_{\tau(i_3)}^\intercal\va_{\tau(i_2)}\va_{\tau(i_2)}^\intercal \va_{\tau(i_1)}\va_{\tau(i_1)}^\intercal +  \dots, \nonumber,
\end{align}
$$\mat{B}_{\tau} := \mat{B}_{\tau(m)},$$
 and
$$\bm{\eta}_{err} :=\frac{\gamma}{n^d} \sum_{i=0}^{m-1} \bm{\eta}_{\tau(m-i)}\left( \mat{I}  - \sum_{j=m-i+1}^{m} \frac{\gamma}{n^d} \va_{\tau(j)}\va_{\tau(j)}^\intercal + \mat{B}_{\tau(i)}\right) \va_{\tau(m-i)}.$$
Further, let us $\mat{B}_{\tau} := \mat{B}_{\tau{m}}$ and $\va_i^T$, $\tau$
The next lemma is the result of a technical rewriting of Step 4 of the algorithm, essentially the same as in the vector case in \cite[Lemma 1]{jeong2025linear}:
\begin{lemma}\label{rem:kacziht-prelim}
    In the notations of Theorem~\ref{thm:main} and with \revised{$\va_i$, $\tau$ and $\mat{B}_{\tau}$ as defined above}, the iterates $\{\tX_i^j\}$ of KaczTIHT satisfy-
\begin{equation}
\label{eq:itr_inner}
\algoname{Vec}(\tX_{m+1}^k-\tX^{\ast}) = \left( \mat{I}-\frac{\gamma}{n^d}\mat{A}^\intercal \mat{A} + \mat{B}_\tau \right) \algoname{Vec}(\tX_1^k-\tX^{\ast}) \nonumber + \bm{\eta}_{err}. 
\end{equation}
\end{lemma}
\begin{proof}
Straightforward check, starting with
\begin{align*}
\algoname{Vec}(\tX_{m+1}^k-\tX^{\ast}) =& \left(\prod_{i = 1}^m (\mat{I} - \frac{\gamma}{n^d} {\va_{\tau(i)} \va_{\tau(i)}^\intercal }) \right)\algoname{Vec}(\tX_1^k-\tX^\ast)  \\ 
&+{\gamma \sum_{i=0}^{m-1}  \frac{\bm{\eta}_{\tau(m-i)}}{n^d} \va_{\tau(m-i)} \prod_{j=0}^{i-1} \left(\mat{I} - \frac{\gamma}{n^d} {\va_{\tau(m-j)} \va^\intercal_{\tau(m-j)}} \right)}\;. 
\end{align*}
\end{proof}
 
We need two more technical lemmas bounding the quantities involved:
\begin{lemma}[Bounding $\mat{B}_{\tau(k)}$]
\label{lem:TRIP_and_bound_B_tau}
Given a measurement operator $\mathcal{A}: \R^{ n \times n \times \dots \times n} \rightarrow \R^m$ and $\mat{B}_{\tau(k)}$ for any $k \in [m]$ defined as above \eqref{eq:b-def} and $0 < \delta < 1$, we have that for for any $k\in[m]$,
$$\lambda \|\mat{B}_{\tau(k)}\|_2 \leq \delta,$$ 
with probability at least $1-m^2 e^{-C \ln^2(m)}$ where $\lambda$ is as in \eqref{eq:parameters}.
\end{lemma}

\begin{proof}
We have that
\[ \| \va_{\tau(i_2)}\va_{\tau(i_2)}^\intercal \va_{\tau(i_1)}\va_{\tau(i_1)}^\intercal \|_2 \leq  \|\va_{\tau(i_2)}\|_2\|\va_{\tau(i_1)}\|_2  |\langle \va_{\tau(i_2)}, \va_{\tau(i_1)} \rangle|\leq n^d |\langle \va_{\tau(i_2)}, \va_{\tau(i_1)} \rangle| \;,\]
and,
\begin{align}
\| \va_{\tau(i_p)}\va_{\tau(i_p)}^\intercal \dots \va_{\tau(i_1)}\va_{\tau(i_1)}^\intercal \|_2 &\leq  \|\va_{\tau(i_p)}\|_2\|\va_{\tau(i_1)}\|_2  |\langle \va_{\tau(i_p)}, \va_{\tau(i_{p-1})} \rangle| \dots |\langle \va_{\tau(i_2)}, \va_{\tau(i_{1})} \rangle| \nonumber\\
&\leq n^d |\langle \va_{\tau(i_p)}, \va_{\tau(i_{p-1})} \rangle| \dots |\langle \va_{\tau(i_2)}, \va_{\tau(i_{1})} \rangle| \; .
\end{align}
Further, it follows from Lemma \ref{lem:orth_row} with $t = \ln(m)$ via a union bound argument that for any $i \neq j \in [m]$, with probability at least $1-m^2 e^{-C \ln^2(m)}$,
$$|\langle\va_i, \va_j \rangle| \leq \ell^2\sqrt{n^d} \ln^2m.$$
Here, $C$ is a positive absolute constant. Then,
\begin{align*} \left\|\frac{\gamma^p}{n^{dp}} \sum_{m-k+1 \leq i_1 < i_2 \dots i_p \leq m} \va_{\tau(i_p)}\va_{\tau(i_p)}^\intercal \dots \va_{\tau(i_2)}\va_{\tau(i_2)}^\intercal \va_{\tau(i_1)}\va_{\tau(i_1)}^\intercal\right\|_2 & \leq \frac{\gamma^p}{n^{dp}} \binom{m}{p} n^d (\ell^4 n^d \ln^4(m))^\frac{p-1}{2} \\ 
&\leq(\gamma m)^p \left( { \ell^4 \ln^4(m)} {n^{-d}} \right)^\frac{p-1}{2}\;.
\end{align*}
Thus, for any $k \in [m]$,
\[  \|\mat{B}_{\tau(k)}\|_2 \leq \sum_{p=2}^{m}  (\gamma m)^p \left( { \ell^4\ln^4(m)} {n^{-d}} \right)^\frac{p-1}{2}  \leq2 \gamma^2 m^2 \ell^2 \ln^2(m) n^{-d/2}.
 \]
Using \eqref{eq:parameters}, we can conclude that 
$\lambda \|\mat{B}_{\tau(k)}\|_2 \leq \delta$.

\end{proof}

\begin{lemma}
\label{lem:TRIP_innerproduct}
Given a measurement operator $\mathcal{A}: \R^{ n \times n \times \dots \times n} \rightarrow \R^m$  that satisfies $(\delta, 3\vr)$-TensorRIP, and HOSVD rank $\vr$ (or, CP rank $r$) tensors $\tX$, $\tY$ and $\tZ \in \R^{n \times n \times \dots \times n}$, we have that
$$ \left\vert \big\langle \left( \mat{I} - \mat{A}^\intercal \mat{A} \right)\algoname{Vec}(\tX - \tZ),   \algoname{Vec}(\tY - \tZ)\big\rangle \right \vert \leq \delta \|\tX - \tZ\|_F \|\tY - \tZ\|_F,$$
where $\mat{A} \in \R^{m \times n^d}$ corresponds to the operator $\mathcal{A}$.
\end{lemma}
\begin{proof}
Without loss of generality, we can assume that $\tZ_1 := \tX - \tZ $ and $\tZ_2 := \tY - \tZ$ are unit norm tensors, let $\vz_1 = \algoname{Vec}(\tZ_1)$ and $\vz_2 = \algoname{Vec}(\tZ_2)$. By the parallelogram law,
$$ \big| \left \langle \mat{A}\vz_1, \mat{A}\vz_2 \right \rangle - \left \langle \vz_1, \vz_2 \right \rangle\big| \leq \frac{1}{4} \big| \|\mat{A} (\vz_1 + \vz_2) \|^2 -   \|\mat{A} (\vz_1 - \vz_2) \|^2  - \|\vz_1 + \vz_2 \|^2 +  \|\vz_1 - \vz_2 \|^2\big|.$$
Now note that, \revision{$\tZ_1 + \tZ_2 = \tX + \tY -2\tZ$ and $\tZ_1 - \tZ_2 = \tX -\tY$} are HOSVD rank \revision{$3\vr$ and $2\vr$} (or equivalently CP rank \revision{$3r$ and $2r$}) tensors.  By using the fact that $\mat{A}$ satisfies $(\delta, \revision{3}\vr)$-TensorRIP we have
$$  \big| \left \langle \mat{A}\vz_1, \mat{A}\vz_2 \right \rangle - \left \langle \vz_1, \vz_2 \right \rangle\big| \leq \frac{\delta}{4} \left( \|\tZ_1 + \tZ_2\|_F^2 + \|\tZ_1 - \tZ_2\|_F^2 \right) \leq \delta,$$
applying the parallelogram law and the unit norm assumption in the last step. Rearranging, this concludes the proof of the lemma.
\end{proof}
With these results in place, we can proceed to prove the main theorem. 

\begin{proof}[Proof of \cref{thm:main}]
From Assumption 2, we know that,
\[ \|\tU^k-\tX_1^{k+1} \|^2_F =\|(\tU^k-\tX^\ast) + (\tX^\ast - \tX_1^{k+1}) \|^2_F \leq (1+\xi)^2 \| \tU^k-\tX^\ast\|_F^2\;.\]
which implies
\begin{gather*}
    \|\tX^\ast -\tX_1^{k+1}\|_F^2 \leq 2\underbrace{\langle \tU^k - \tX^\ast , \tX_1^{k+1} - \tX^\ast \rangle}_{\text{\tt{Term 1}}} + \ (2\xi + \xi^2)\underbrace{\|\tU^k - \tX^\ast\|_F^2}_{\text{\tt{Term 2}}}\;.
\end{gather*}
From Lemma~\ref{rem:kacziht-prelim}, we know that 
$$
\algoname{Vec}(\tX_{m+1}^k-\tX_1^k) = \left(- \frac{\gamma} {n^d}\mat{A}^\intercal \mat{A} + \mat{B}_\tau\right) \algoname{Vec}(\tX_1^k-\tX^{\ast}) \nonumber + \bm{\eta}_{err}. $$
Also, $  \tU^k = \tX_1^{k} + \lambda(\tX_{m+1}^k - \tX_1^k)$, giving us
$$\tU_k - \tX^* = \left( \mat{I} - \lambda \frac{\gamma}{n^d}\mat{A}^\intercal \mat{A} + \lambda\mat{B}_{\tau} \right) (\tX_1^k - \tX^\ast) + \revised{\lambda \bm{\eta}_{err}}$$

By using \cref{lem:trip_hosvd} (or equivalently \cref{theorem:trip_cp}), we get that under our hypothesis, the measurement operator $\frac{1}{\sqrt{m}}\mat{A}$ also satisfies $(\delta, 3\vr)$-TensorRIP (or equivalently $(\delta, 3r)$-TensorRIP for low rank CP recovery) with probability at least $1 -  e^{-C\ln^2m}$ for some $C > 0$. Together with \cref{lem:TRIP_and_bound_B_tau}, we get that with probability at least $1 - (m^2+1)e^{-C\ln^2m}$
\begin{align*}
   &\|\tU^k - \tX^{\ast}\|_F = \left\|\left(\mat{I} - \frac{1}{m} {\mat{A}}^\intercal \mat{A} + \lambda \mat{B}_\tau\right)\algoname{Vec}(\tX_1^k - \tX^{\ast}) + \lambda\bm{\eta}_{err}\right\|_2  \quad (\text{Since }\lambda = \frac{n^d}{m\gamma} ) \\
    &\leq \|\tX_1^k - \tX^{\ast}\|_F + \left\| \frac{1}{\sqrt{m}}  {\mat{A}}^\intercal \right\|_2 \left\| \frac{1}{\sqrt{m}}\mat{A}\algoname{Vec}(\tX_1^k - \tX^{\ast})\right\|_2 + \|\lambda \mat{B}_\tau\|_2\|\tX_1^k - \tX^{\ast}\|_F + \|\lambda\bm{\eta}_{err}\|_2  \\
    &\leq \left( 1 + \delta + \left\|\frac{1}{\sqrt{m}}\mat{A}\right\|_2 \sqrt{1 + \delta}\right)\|\tX_1^k - \tX^{\ast}\|_F + \|\lambda\bm{\eta}_{err}\|_2\;. \quad (\text{Using \cref{lem:TRIP_and_bound_B_tau}})
\end{align*}
and
\[ \text{\tt{Term 2}} = \|\tU^k - \tX^{\ast}\|_F^2  \leq 2 \left( 1 + \delta + \left\|\frac{1}{\sqrt{m}}\mat{A}\right\|_{2} \sqrt{1 + \delta}\right)^2\|\tX_1^k - \tX^{\ast}\|_F^2 + 2 \|\lambda\bm{\eta}_{err}\|_2^2 ;.\]
And, for \text{\tt{Term 1}} we get that, using Lemma~\ref{lem:TRIP_innerproduct},
\begin{align*}
    &\langle \tU^k-\tX^\ast , \tX_1^{k+1} - \tX^\ast \rangle \\
    &= \left\langle \left(\mat{I} - \frac{1}{m} \mat{A}^\intercal \mat{A} \right) \algoname{Vec}(\tX_1^{k+1} - \tX^{\ast}), \algoname{Vec}(\tX_1^k - \tX^{\ast}) \right\rangle \\
    & \quad\quad\quad\quad + \left \langle \lambda \mat{B}_\tau \algoname{Vec}(\tX_1^{k+1} - \tX^{\ast}), \algoname{Vec}(\tX_1^k - \tX^{\ast}) \right\rangle + \left\langle \lambda\bm{\eta}_{err}, \algoname{Vec}(\tX_1^{k+1} - \tX^{\ast}) \right\rangle \\
    &\leq \delta\|\tX_1^{k+1} -\tX^\ast\|_F\|\tX_1^k-\tX^{\ast}\|_F +  \|\lambda\mat{B}_\tau \|_2 \|\tX_1^{k+1} - \tX^{\ast}\|_F\|\tX_1^k - \tX^{\ast}\|_F  + \|\lambda\bm{\eta}_{err}\|_2\|\tX_1^{k+1} - \tX^{\ast}\|_F. 
\end{align*}
Putting this all together with Lemma~\ref{lem:TRIP_and_bound_B_tau}, we get 
\begin{align*}\|\tX^{\ast} -\tX_1^{k+1}\|_F^2 &\leq 
 2(2\xi+\xi^2)\left(\left( 1 + \delta + \left\|\frac{1}{\sqrt{m}}\mat{A}\right\|_2 \sqrt{1 + \delta}\right)^2\|\tX_1^k - \tX^{\ast}\|_F^2 + \|\lambda\bm{\eta}_{err}\|_2^2\right)\\ & \quad \quad \quad \quad + 4\delta \|\tX_1^{k+1} -\tX^\ast\|_F\|\tX_1^k-\tX^{\ast}\|_F + 2\|\lambda\bm{\eta}_{err}\|_2 \|\tX_1^{k+1}-\tX^{\ast}\|_F.
\end{align*}
This implies that there exist $\alpha, \beta, \prod \in [0,1]$ such that
\begin{align}
\label{eq:step1}
    &(1-\alpha - \beta - \prod)\|\tX_1^{k+1} - \tX^{\ast}\|_F^2 \leq 4\delta\|\tX_1^{k+1}-\tX^{\ast}\|_F\|\tX_1^k - \tX\|_F\;, \\
    \label{eq:step2}
    &\alpha \| \tX_1^{k+1} - \tX^\ast\|_F^2 \leq {2}(2\xi + \xi^2)\left( 1 + \delta + \left\|\frac{1}{\sqrt{m}}\mat{A}\right\|_2 \sqrt{1 + \delta}\right)^2 \|\tX^k_1 - \tX^{\ast}\|_F^2\;, \\
    \label{eq:step3}
    &\beta \| \tX_1^{k+1} - \tX^{\ast}\|_F^2 \leq {2}(2\xi + \xi^2) \|\lambda\bm{\eta}_{err}\|_2^2 \;,\\
    \label{eq:step4}
    &\prod \| \tX_1^{k+1} - \tX^{\ast}\|_F^2 \leq 2\|\lambda\bm{\eta}_{err}\|_2 \|\tX_1^{k+1}-\tX^{\ast}\|_F\;. \end{align}
\revised{By combining, \eqref{eq:step1} and \eqref{eq:step4}, we get,
    \begin{equation*}
    (1-\alpha-\beta)\|\tX_1^{k+1} - \tX^{\ast}\|_F \leq 4\delta\|\tX_1^{k} - \tX \|_F + 2 \|\lambda\bm{\eta}_{err}\|_2
    \end{equation*}
    Now, by utilizing \eqref{eq:step2} and \eqref{eq:step3} by taking the square root on both sides, we have-
        \begin{align*}
\left({1 - \alpha -\beta + \sqrt{\alpha} + \sqrt{\beta}  }\right) \|\tX_1^{k+1} &- \tX^\ast\|_F \\
\leq& \  \left[4\delta + \sqrt{4\xi + 2\xi^2}\left( 1 + \delta + \left\|\frac{1}{\sqrt{m}}\mat{A}\right\|_2 \sqrt{1 + \delta}\right)\right]\|\tX^k_1 - \tX^\ast\|_{\revised{F}} \\
    &+ \left({\sqrt{4\xi+2\xi^2} + 2}\right)\|\lambda\bm{\eta}_{err}\|_2\;.
\end{align*}}
    Then,
    \begin{align*}
 \|\tX_1^{k+1} - \tX^\ast\|_F \leq& \ \frac{1}{{1 - \alpha -\beta + \sqrt{\alpha} + \sqrt{\beta}  }} \left[4\delta + \sqrt{4\xi + 2\xi^2}\left( 1 + \delta + \left\|\frac{1}{\sqrt{m}}\mat{A}\right\|_2 \sqrt{1 + \delta}\right)\right]\|\tX^k_1 - \tX^\ast\|_{\revised{F}} \\
    &+ \frac{\sqrt{4\xi+2\xi^2} + 2}{{1 - \alpha -\beta + \sqrt{\alpha} + \sqrt{\beta}  }}\|\lambda\bm{\eta}_{err}\|_2\;.
\end{align*}
Note that $\alpha, \beta \in [0,1]$ the function $(1 - \alpha - \beta + \sqrt{\beta} +\sqrt{\alpha})^{-1}$ is positive and strictly bounded by $1$, also let $\sigma := 1 + \delta + \|\frac{1}{\sqrt{m}}\mat{A}\|_{ 2} \sqrt{1 + \delta} > 1$. We know from \eqref{eq:extra_assmpn_4}, that the thresholding approximation is $\xi$-accurate where $ \xi \leq \frac{\delta^2}{5\sigma^2}$. So, we have 
\begin{align*}
\label{eq:kziht_itr}
\|\tX_1^{k+1} - \tX^\ast\|_F &\leq  ( 4\delta + \sigma\sqrt{4\xi + 2\xi^2})\|\tX^k_1 - \tX^\ast\|_2 +(\sqrt{4\xi+2\xi^2} + 2)\|\lambda\bm{\eta}_{err}\|_2 \\
&\le 5\delta \|\tX_1^k-\tX^{\ast}\|_F + (\delta + 2)\|\lambda\bm{\eta}_{err}\|_2 \\
&\leq (5\delta)^{k+1} \|\tX^{\ast}\|_F + \frac{\delta + 2}{1-5\delta}\|\lambda\bm{\eta}_{err}\|_2.
\end{align*}

Finally, we bound the $\bm{\eta}_{err}$ term as follows: from Lemma~\ref{rem:kacziht-prelim},
\begin{align*}
    \|\lambda\bm{\eta_{err}}\|_2 &\leq \left\|\frac{1}{m} \sum_{i=0}^{m-1} \bm{\eta}_{\tau(m-i)}\left( \mat{I}  - \sum_{j=m-i+1}^{m} \frac{\gamma}{n^d} \va_{\tau(j)}\va_{\tau(j)}^\intercal + \mat{B}_{\tau(i)}\right) \va_{\tau(m-i)}\right\|_2\\
    &\leq \frac{1}{m}\left\|\sum_{i=1}^{m} \eta_i \va_i \right\|_2 +\frac{\gamma}{n^d}  \sum_{i=0}^{m-1} |\bm{\eta}_{\tau(m-i)}| \underset{j}{\max} \; \|\va_{\tau(j)}\va_{\tau(j)}^\intercal\|_2 \underset{i}{\max}  \| \va_{\tau(m-i)}\|_2 \\
    & \quad\quad\quad\quad\quad\quad\quad\quad\quad\quad\quad +\frac{1}{m}\sum_{i=0}^{m-1} |\bm{\eta}_{\tau(m-i)}| \underset{i}{\max} \|\mat{B}_{\tau(i)}\|_2  \| \va_{\tau(m-i)}\|_2 \\
    & \revised{\underset{(a)}{\leq} {\frac{1}{m}}\|\mat{A}^\intercal \bm{\eta\|_2}   + {\gamma}  \|\bm{\eta}\|_1 \underset{i}{\max}\| \va_{\tau(m-i)}\|_2 +\frac{1}{m} \|\bm{\eta}\|_1 \underset{i}{\max} \|\mat{B}_{\tau(i)}\|_2 \ \underset{i}{\max}\| \va_{\tau(m-i)}\|_2} \\
    & \revised{\underset{(b)}{\leq} {\frac{1}{m}}\|\mat{A}^\intercal \bm{\eta\|_2}   +   \|\bm{\eta}\|_1 \sqrt{n^d}\left( {\gamma}  +\frac{1}{m}  \underset{i}{\max} \ \|\mat{B}_{\tau(i)}\|_2   \right)} \\
    & \underset{(c)}{\leq}  \revised{\frac{1}{m}}\|\mat{A}^\intercal \bm{\eta\|_2} + \frac{1}{m} \left(\frac{\delta}{2(\ell\ln m)^2} + \frac{\delta^2}{2n^d(\ell\ln m )^2}\right)\|\bm{\eta}\|_1.
 \end{align*}
\revised{Here, steps $(a)$ and $(b)$ both follow from the fact that $\|\va_{\tau(j)}\| = \sqrt{n^d}$ (Assumption 1) and $(c)$ follows from \cref{lem:TRIP_and_bound_B_tau} and \eqref{eq:parameters}.} 
Finally we can conclude the proof with
\[ \|\tX_1^{k+1} - \tX\|_F \leq (5\delta)^{k+1} \|\tX^\ast\|_F + \frac{1}{m}\frac{\delta + 2}{1-5\delta}\left(\left( \frac{\delta}{2(\ell\ln m)^2} + \frac{\delta^2}{2n^d(\ell\ln m )^2}\right)\|\bm{\eta}\|_1 + 
    \|\mat{A}^\intercal \bm{\eta}\|_2 \right).\]
\end{proof}

\begin{remark}[KaczTIHT for BOS measurement matrix]
\label{rem:bos} 
Note that when the measurement matrix $\mat{A}$, corresponding to the linear measurement operator $\mathcal{A}$, is sampled from a bounded orthogonal system, that is, the rows of $\mat{A}$ are orthogonal, KaczTIHT is equivalent to the standard TIHT iteration (similarly to the vector case in \cite{jeong2025linear}). Indeed, we can simplify \eqref{eq:itr_inner} to
\begin{equation*}
\label{eqn:iteratios_kz}
\algoname{Vec}\left(\tX^k_{m+1} - \tX^{\ast}\right) = \left( \mat{I} - \gamma \sum_{i=1}^{m} \frac{\va_{\tau(i)} \va_{\tau(i)}^\intercal}{n^d} \right) \algoname{Vec}(\tX_1^k - \tX^{\ast}) + \sum_{i=0}^{m-1}\bm{\eta}_{\tau(m-i)} \gamma \frac{\va_{\tau(m-i)}}{n^d} \;.
\end{equation*}
For the $\|\va_j\|^2 = n^d$ for all $j = 1,\dots m$ the step-sizes chosen to be $\gamma = n^d/m$ and $\lambda = 1$, 
$$ \gamma \sum_{i=1}^{m} \frac{\va_{\tau(i)} \va_{\tau(i)}^\intercal}{\|\va_{\tau(i)}\|_2^2} = \frac{n^d}{m} \sum_{i=1}^{m} \frac{\va_{\tau(i)} \va_{\tau(i)}^\intercal}{n^d} = \frac{1}{m}\mat{A}^\intercal \mat{A}\; .$$
In this case, the iterates of KaczTIHT satisfy
\begin{gather*}
\tX_{m+1}^k = \tX_1^k + \frac{1}{m} \mat{A}^\intercal \mat{A} \algoname{Vec}(\tX^{\ast} - \tX_1^k) + \frac{1}{m} \mat{A}^\intercal \bm{\eta}  = \tX_1^k +  \frac{1}{m} \mat{A}^\intercal(\vb- \mat{A}\tX_1^k)),\\
 \algoname{Vec}(\tX_1^{k+1}) = T_\mathbf{r}( \algoname{Vec}(\tX^k_{m+1})) = T_\mathbf{r}\left(\tX_1^k + \frac{1}{m} \mat{A}^\intercal(\vb - \mat{A} \algoname{Vec}(\tX_1^k))\right).\end{gather*}
This coincides with the iterations of TIHT. 
\end{remark}

\section{Experiments}
\label{section:experiments}
\revised{In this section, we provide  numerical experiments that demonstrate the performance of TrimTIHT and KaczTIHT methods on a variety of synthetic datasets and real-world video data. Our main goals are to \revision{empirically} investigate suitable  parameter choices for the methods   and \revision{to demonstrate} how the proposed methods improve the performance of the standard TIHT method.} The code can be found at \url{https://github.com/shambhavi-suri/Low-Rank-Tensor-Recovery}. All the experiments were run on Adroit, a Beowulf cluster, using 1 core with \revised{7}0 GB RAM.

\textbf{Choosing Parameters}: For KaczTIHT, we set $\lambda = 1$ and $\gamma = \frac{n^d}{m}$. While theory mandates a small choice of $\gamma$, we observe that empirically the performance did not depend on the values of these parameters as long as the product $\gamma \lambda = \frac{n^d}{m}$. For both TIHT and TrimTIHT, we consider \revision{a} step size \revision{of} $\mu = 1$. This ensures a fair comparison and enables us to study the effect of just trimming on the convergence rate. Further, \revised{for the TrimTIHT experiments with $15\times 15 \times 15$ tensors}, we  select $m_{trim}$ by optimizing over a set of 5 values that range from 5 to 80, over 5 random sample runs. \revised{For the experiments, with larger tensors, we choose the parameters by running around 2 sample runs with less than 3 parameters.}

\subsection{Synthetic Data} 
We consider the recovery of low HOSVD and CP rank tensors which have been compressed using two different classes of measurement ensembles: entry-wise independent Gaussian $\revised{N}(0,1)$ measurements and face-splitting product of entry-wise independent Gaussian $\revised{N}(0,1)$ matrices. Low HOSVD rank $\vr = (r,r,\dots,r)$ tensors are generated via a modewise product of a randomly sampled core tensor $\tC$ and orthogonal factor matrices $\mat{U}_1, \mat{U}_2\ \text{and} \ \mat{U_3}$. More precisely, $\tC$ is generated by sampling entry-wise from a $\revised{N{(0.1,1)}}$ distribution and $\mat{U}_i$'s correspond to the top $r$ left singular vectors corresponding to a random matrix in $\R^{15 \times 15}$ whose entries are sampled i.i.d from $\revised{N{(0.1,1)}}$. \revised{Similarly, a random CP rank $r$ tensor, $\tX = \sum_{j=1}^{r} \vx_{1j} \out \vx_{2j} \dots \out \vx_{dj}$ is generated by randomly sampling each of the factors $\vx_{ij}$ in its decomposition. Here, each coordinate of $\vx_{ij}$ is sampled i.i.d from a random Gaussian ${N}(0.1,1)$ distribution.}

\subsubsection{Gaussian Measurements}
\label{subsection:gaussian_numerica[} 
Gaussian $N(0,1)$ matrices with independent entries
have good RIP properties, and the standard low-rank recovery methods, such as TIHT, are expected to work. Yet, we demonstrate the advantage of the proposed adaptive methods, \revised{TrimTIHT and KaczTIHT,} that is especially prominent when the rank gets higher.

\begin{figure}[!ht]
\begin{center} \begin{subfigure}{0.32\textwidth}
     \includegraphics[width=\textwidth]{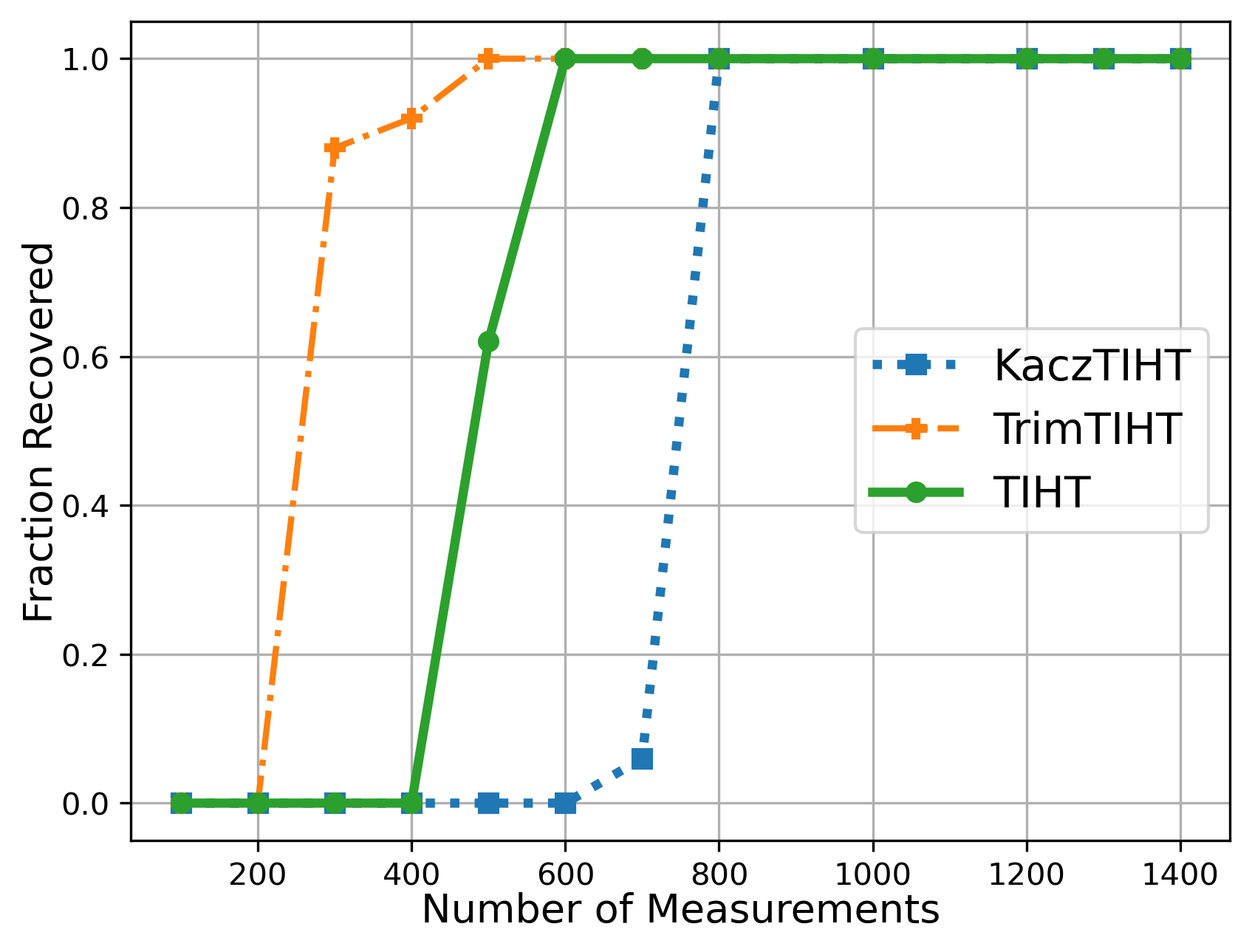}
     \label{fig:recovery rank 2}
     \vspace*{-.3cm}
     \caption{HOSVD Rank$(2,2,2)$}
 \end{subfigure}
  \hfill
\begin{subfigure}{0.32\textwidth}
\includegraphics[width=\textwidth]{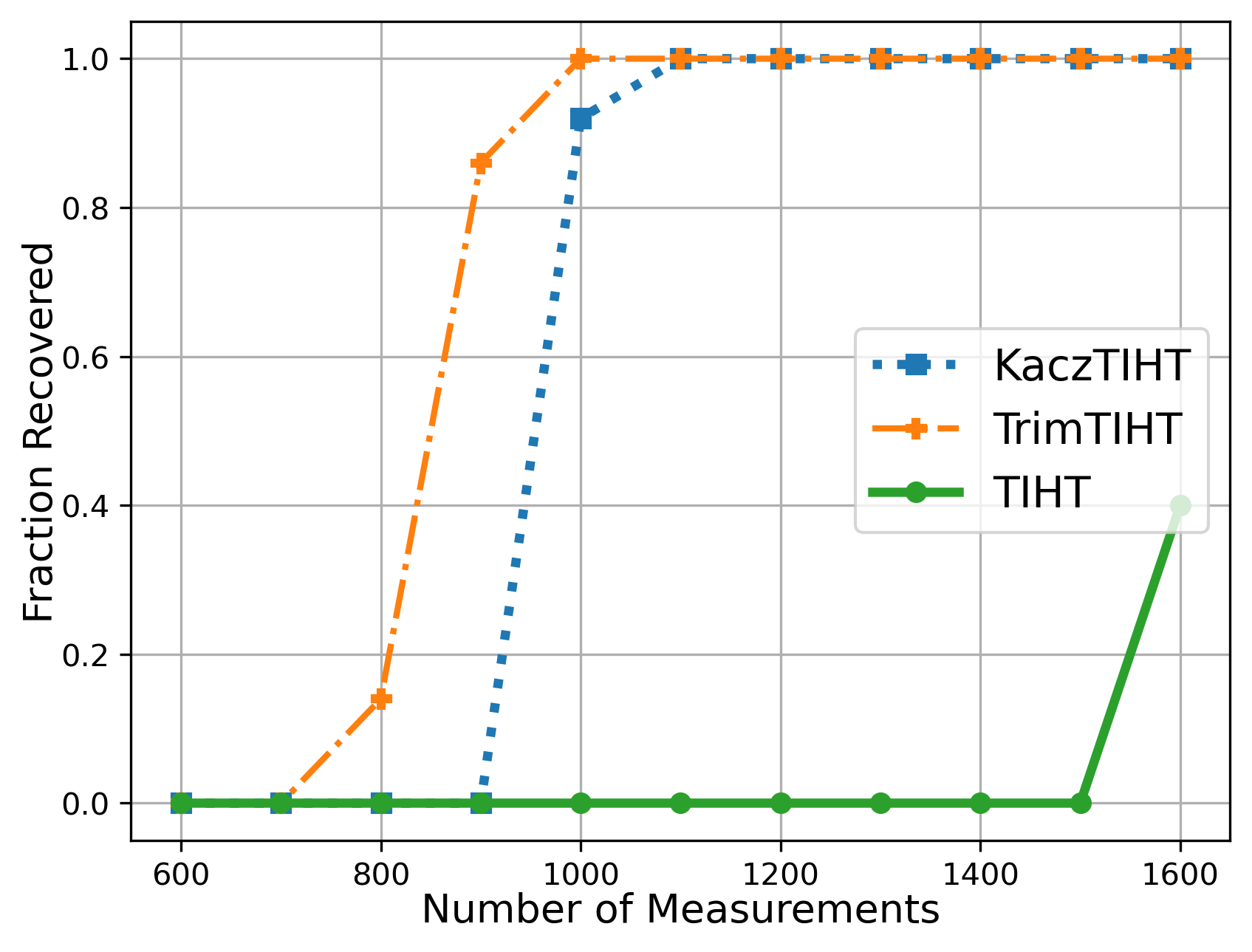}    \label{fig:recovery_rank_r}
     \vspace*{-.3cm}
     \caption{HOSVD Rank $(5,5,5)$}
 \end{subfigure}
\hfill
  \begin{subfigure}{0.32\textwidth}  \includegraphics[width=\textwidth]{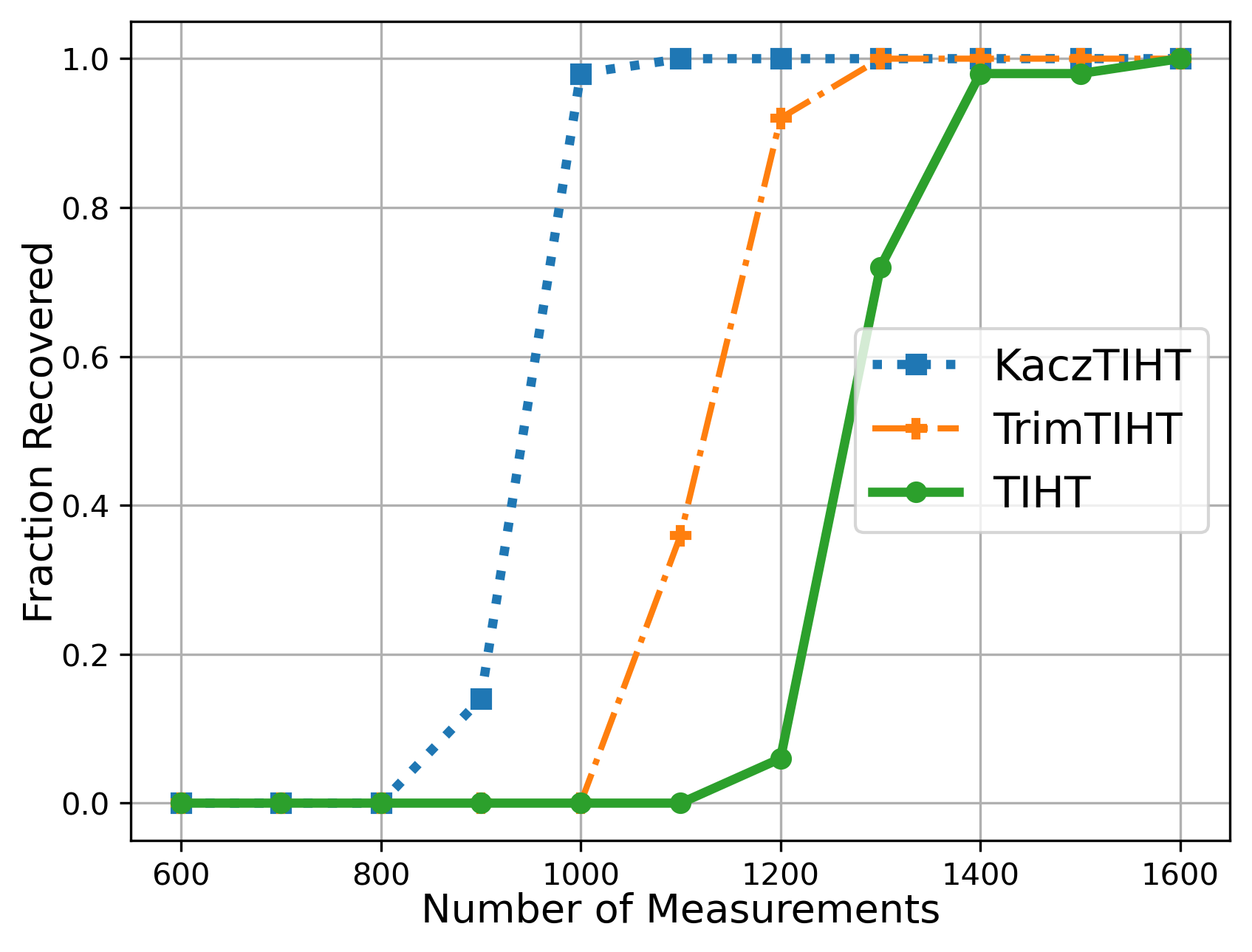} \label{fig:recovery_CP_rank_5}
     \vspace*{-.3cm}
     \caption{CP Rank $5$}
 \end{subfigure}
 \caption{\revised{Recovery \revised{of a $15 \times 15 \times 15$ tensor} from Gaussian measurements: Fraction of $50$ artificially generated random tensors of a certain rank, recovered successfully using TIHT, TrimTIHT\revision{,} and KaczTIHT at different compression levels. The proposed methods require \revision{fewer} measurements for successful recovery, especially for the higher ranks. 
 }}
 \label{fig:Gaussian Recovery}
 \end{center}
\end{figure}
First, in  Figure \ref{fig:Gaussian Recovery}(\texttt{A} and \texttt{B}), we consider the recovery of $50$ HOSVD rank $(2,2,2) $ and $(5,5,5)$ tensors in $\R^{15 \times 15 \times 15}$ which have been compressed to different dimensions. We say that a tensor has been recovered, if the relative error is less than $10^{-4}$ after 200 iterations. Relative error here corresponds to $\|\tX - \tY\|_F/{\|\tX\|_F}$ where $\tY$ is an estimate of $\tX$. We then record the fraction recovered at each compression level. \revision{The} TrimTIHT method performs the best for both these ranks, and KaczTIHT outperforms vanilla TIHT at a higher rank, when TIHT recovery essentially fails.
Similarly, the same experiment was repeated with CP rank $5$ tensors (Figure \ref{fig:Gaussian Recovery}(\texttt{C})). \revision{Here,} we observe that both KaczTIHT and TrimTIHT outperform TIHT, with the former being more stable.

Subsequently, in Figure \ref{fig:convergence_gaussian}, we compare the convergence of the considered IHT-based recovery methods. 
We \revision{observe} that for both HOSVD and CP rank, KaczTIHT and TrimTIHT \revision{outperform} TIHT. \revised{For CP low rank,} KaczTIHT converges faster, especially for higher rank tensors, and TIHT starts diverging.
\begin{figure}
\label{figure:fixed_rank_gaussian}
\begin{center}
 \begin{subfigure}{0.45\textwidth}     \includegraphics[width=\textwidth]{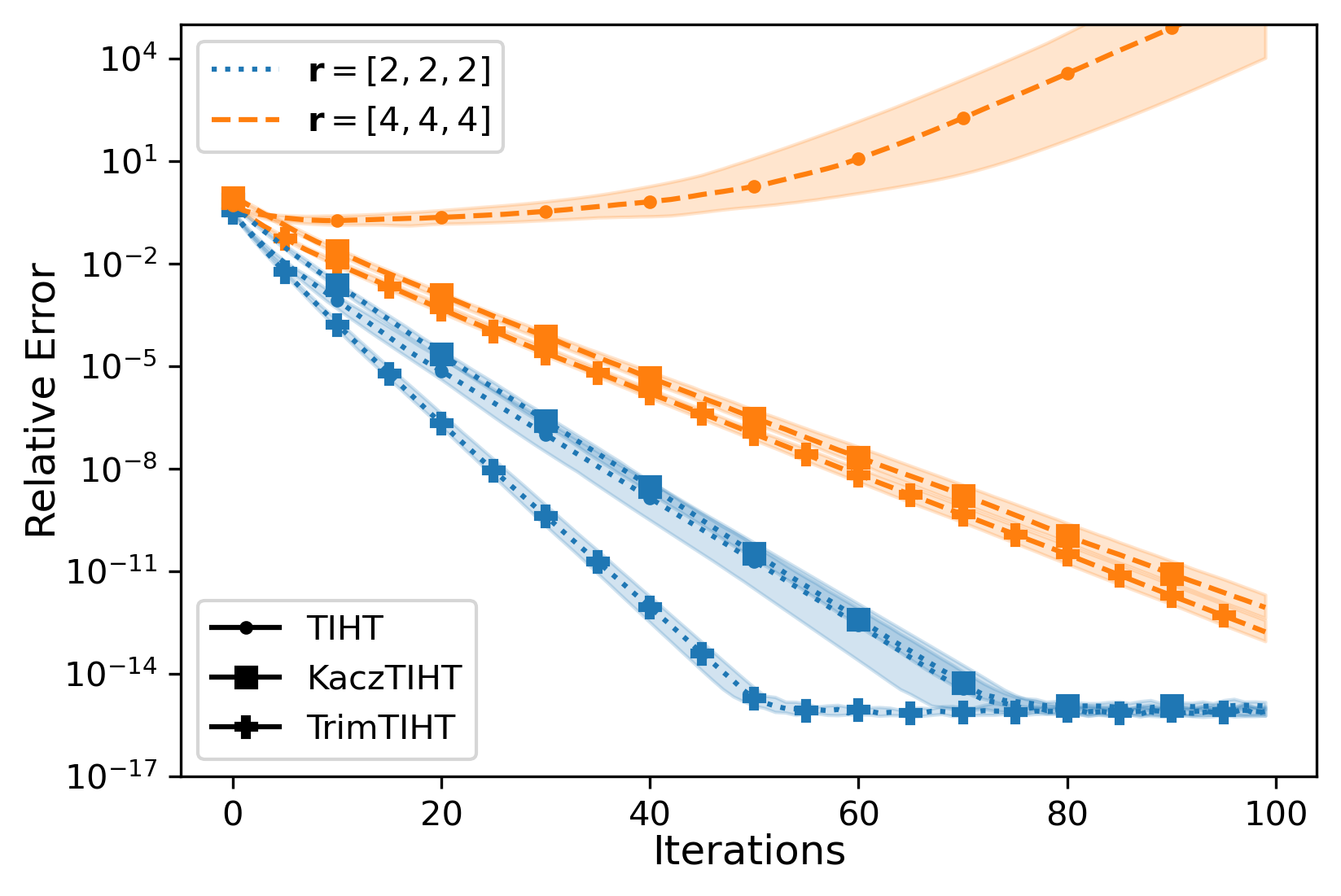}
     \label{fig:Convergence-Gaussian-HOSVD}
     \vspace*{-.3cm}
     \caption{HOSVD rank ${\bf r}$, $m = 1000$ measurements}
 \end{subfigure}
  \hfill
\begin{subfigure}{0.45\textwidth}   \includegraphics[width=\textwidth]{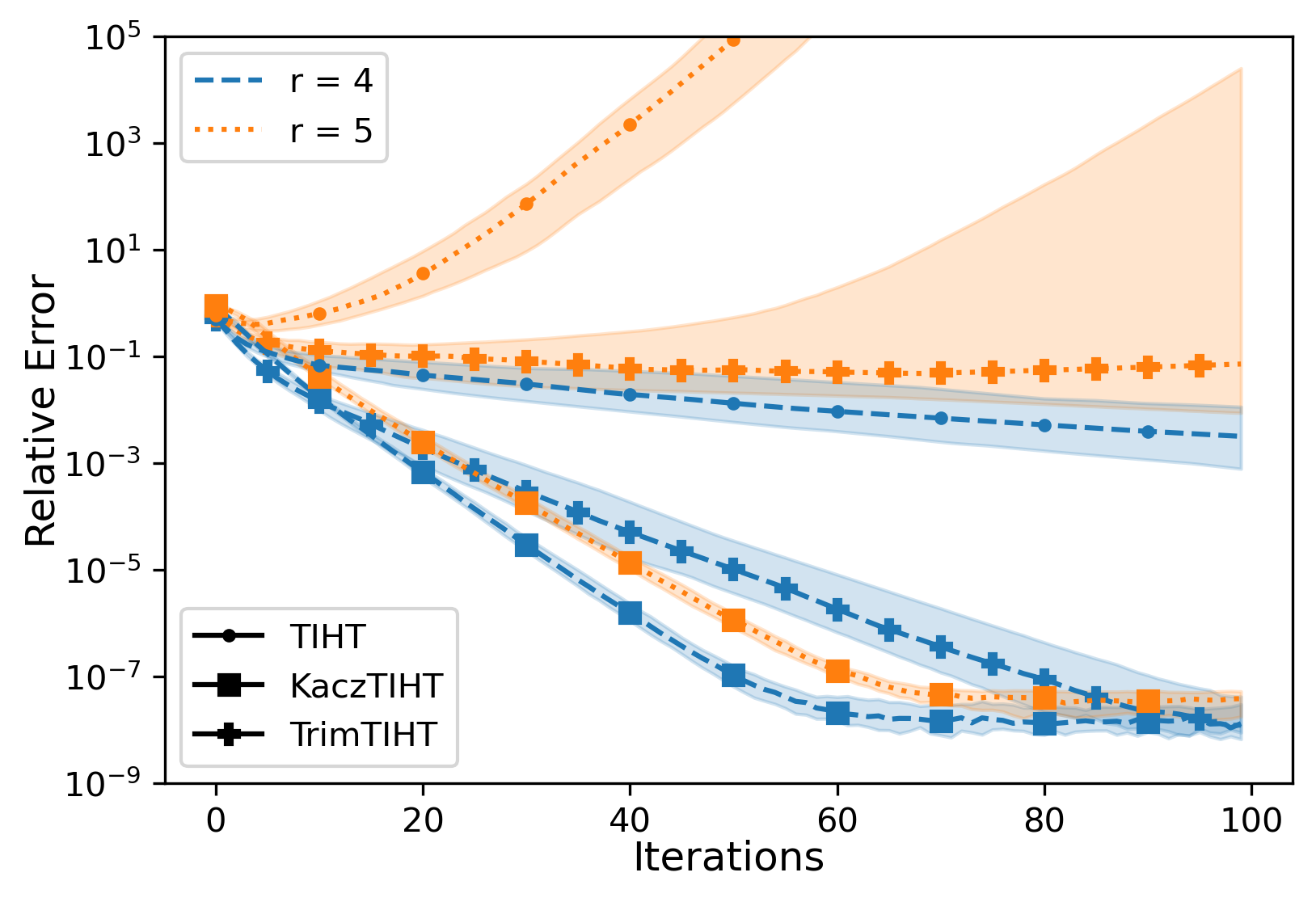}
     \label{fig:convergence-Gaussian-CP}
     \vspace*{-.3cm}
     \caption{CP rank $r$, $m = 1000$ measurements}
 \end{subfigure}
 \end{center}

  \caption{Recovery \revised{of a $15 \times 15 \times 15$ tensor} from  Gaussian measurements: Relative error dynamic during the recovery with TIHT, KaczTIHT\revision{,} and TrimTIHT of  the tensors of different low rank. 
  The lines correspond to the median over $40$ sample runs, and the band represents the inter-quartile range.}
    \label{fig:convergence_gaussian}
\end{figure}

\subsubsection{Face-Splitting Measurements}
\revised{Next, we consider the compression of random low-rank tensors using face-splitting product of entry-wise Gaussian $N(0,1)$ matrices. On the same size $15 \times 15 \times 15$, such measurements require $75$ times less memory than their independent Gaussian counterparts considered in Section~\ref{subsection:gaussian_numerica[}.} However, face-splitting measurements do not have good RIP properties (recall Proposition~\ref{prop:db_friendly_trip}), and empirically we see that the TIHT method does not recover the tensors from any compression ratio we considered (up to $40\%$ of data). However, we show that the proposed methods enable the recovery.

In  Figures \ref{img:recovery_face-split} (\texttt{A} and \texttt{B}), we consider the recovery of $50$ HOSVD rank $(2,2,2) $ and $(5,5,5)$ tensors and CP rank $5$ tensors in $\R^{15\times15\times 15}$.  
 Again, we say that the tensor has been recovered if the relative error is less than $10^{-4}$ after 200 iterations and record the fraction recovered at each compression level. 
 Here, TIHT fails for all considered cases, whereas TrimTIHT and KaczTIHT are able to recover tensors from highly compressed measurements. Then, in Figure \ref{img:face_split_convergence}, we study the convergence of these methods \revision{for a fixed number of} measurements. Again, we see that for both HOSVD and CP rank, KaczTIHT and TrimTIHT converge faster than TIHT. Additionally, KaczTIHT tends to converge faster than TrimTIHT \revised{in terms of the number of iterations (Figure \ref{img:face_split_convergence}) but at least for HOSVD rank it needs more measurements to converge (Figures \ref{img:recovery_face-split}). }
 \begin{figure}[!ht]
\begin{center}
 \begin{subfigure}{0.32\textwidth}
     \includegraphics[width=\textwidth]{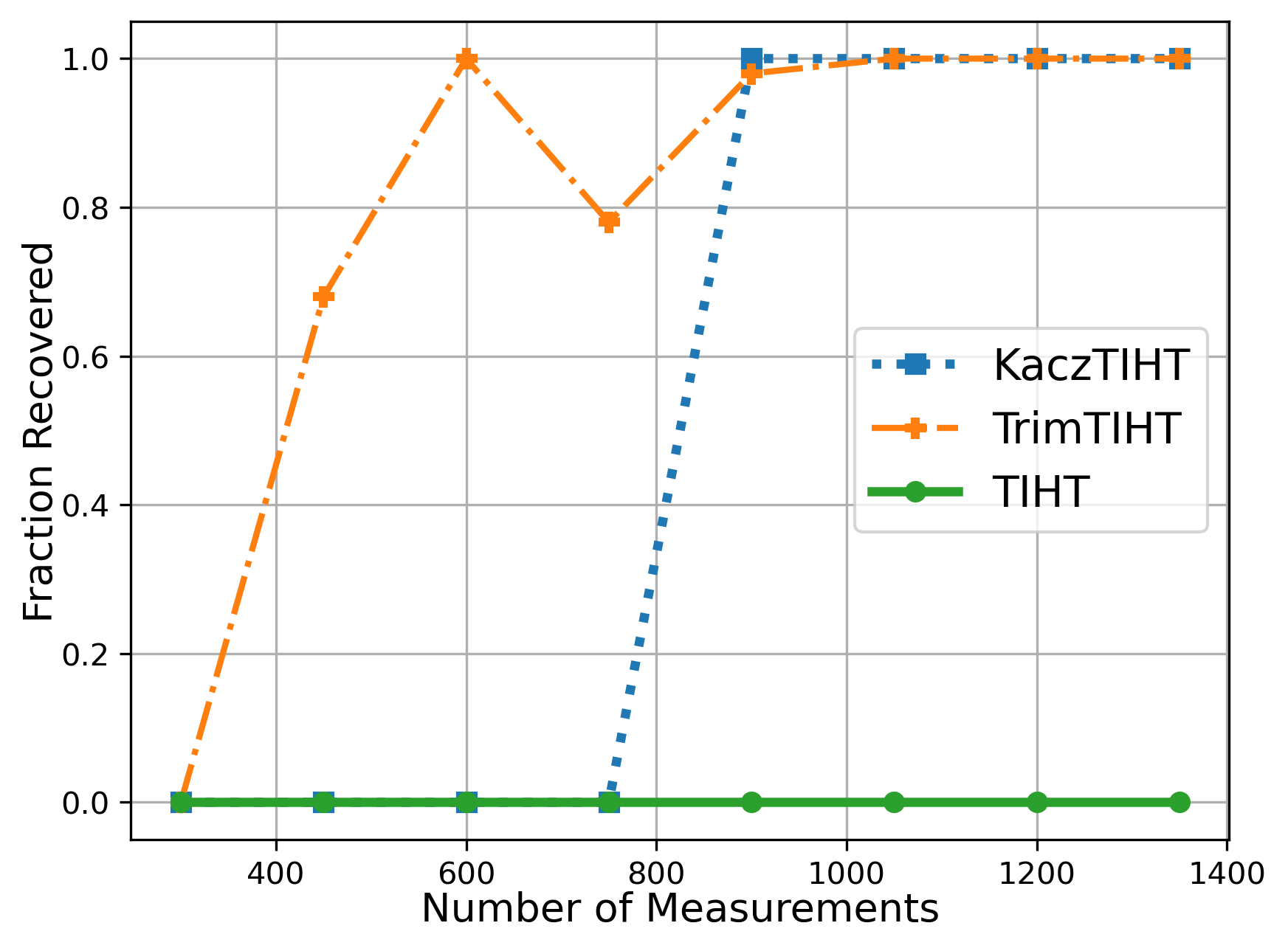}
\vspace*{-.3cm}
     \caption{{HOSVD Rank (2,2,2)}}
 \end{subfigure}
  \hspace{0.2em}
\begin{subfigure}{0.32\textwidth}
     \includegraphics[width=\textwidth]{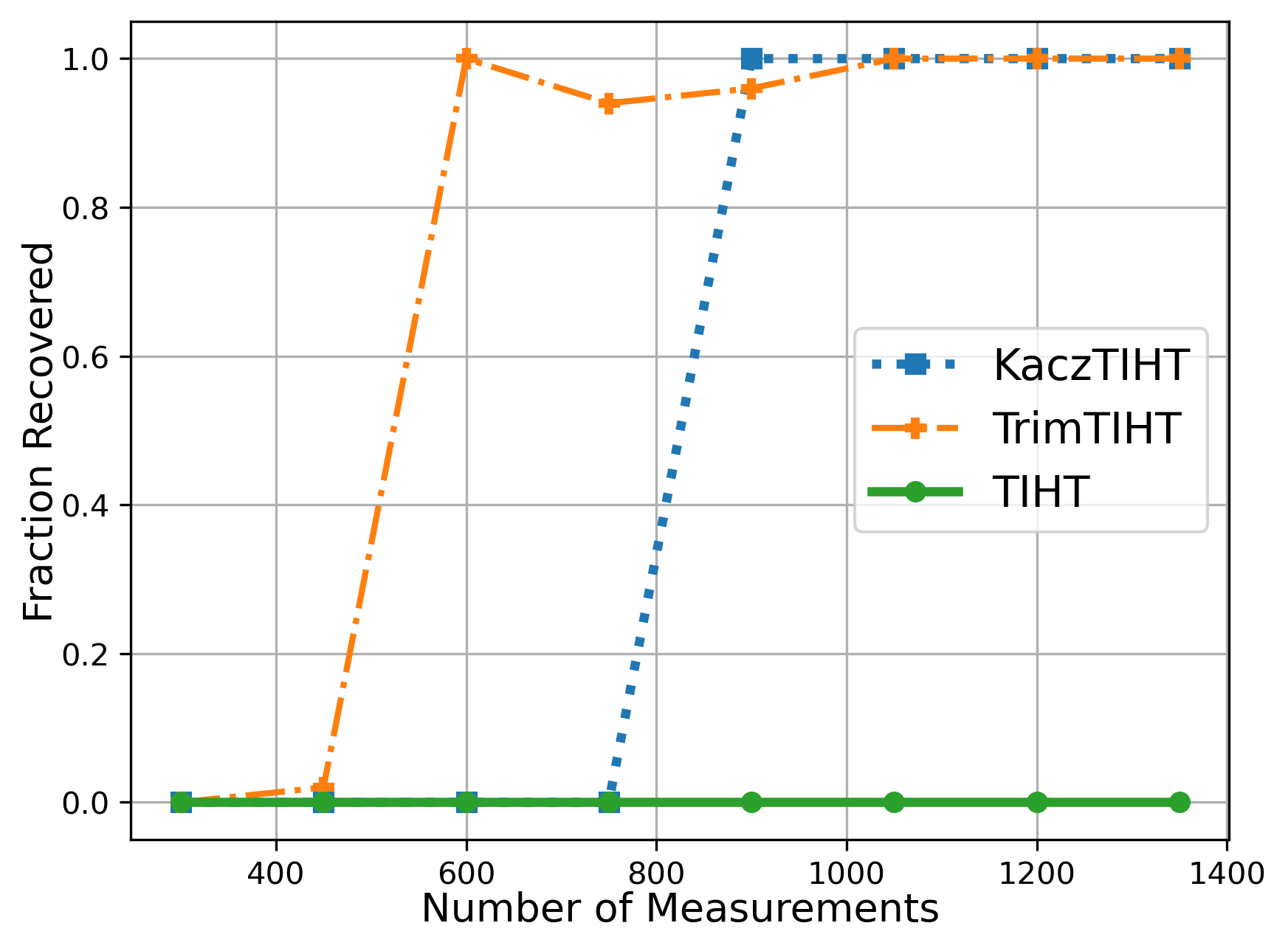}
     \vspace*{-.3cm}
     \caption{{HOSVD Rank (4,4,4)}}
 \end{subfigure}
   \hspace{0.2em}
\begin{subfigure}{0.32\textwidth}
     \includegraphics[width=\textwidth]{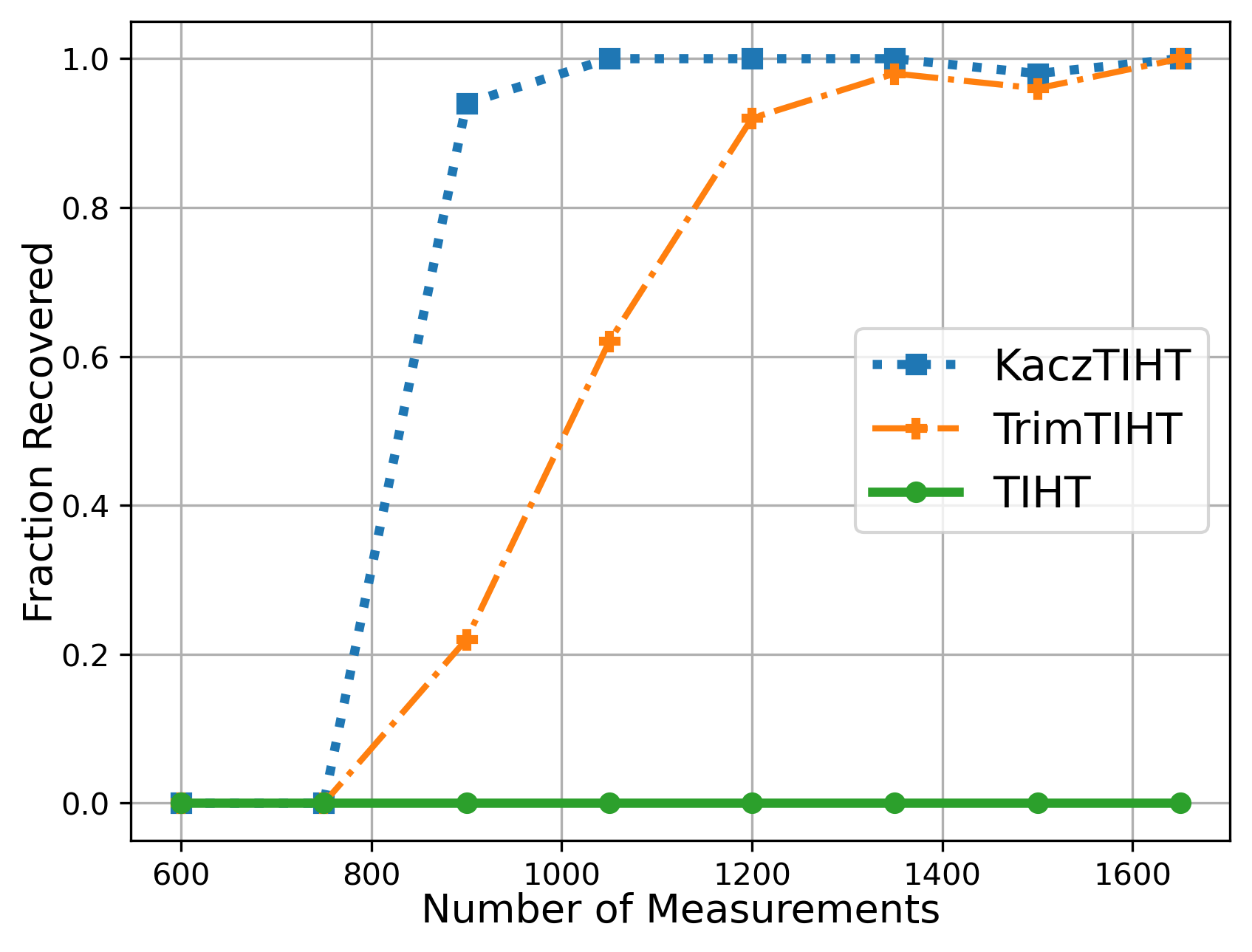}
     \vspace*{-.3cm}
     \caption{{CP Rank 3}}
 \end{subfigure}
 \caption{\revised{Recovery \revised{of a $15 \times 15 \times 15$ tensor} from face-splitting measurements: Fraction of $50$ artificially generated random tensors of a certain rank, recovered successfully using TIHT, TrimTIHT, and KaczTIHT from different levels of compression.
 }}
  \label{img:recovery_face-split}
 \end{center}
\end{figure}

\begin{figure}[!ht]
\begin{center}
 \begin{subfigure}{0.45\textwidth}
     \includegraphics[width=\textwidth]{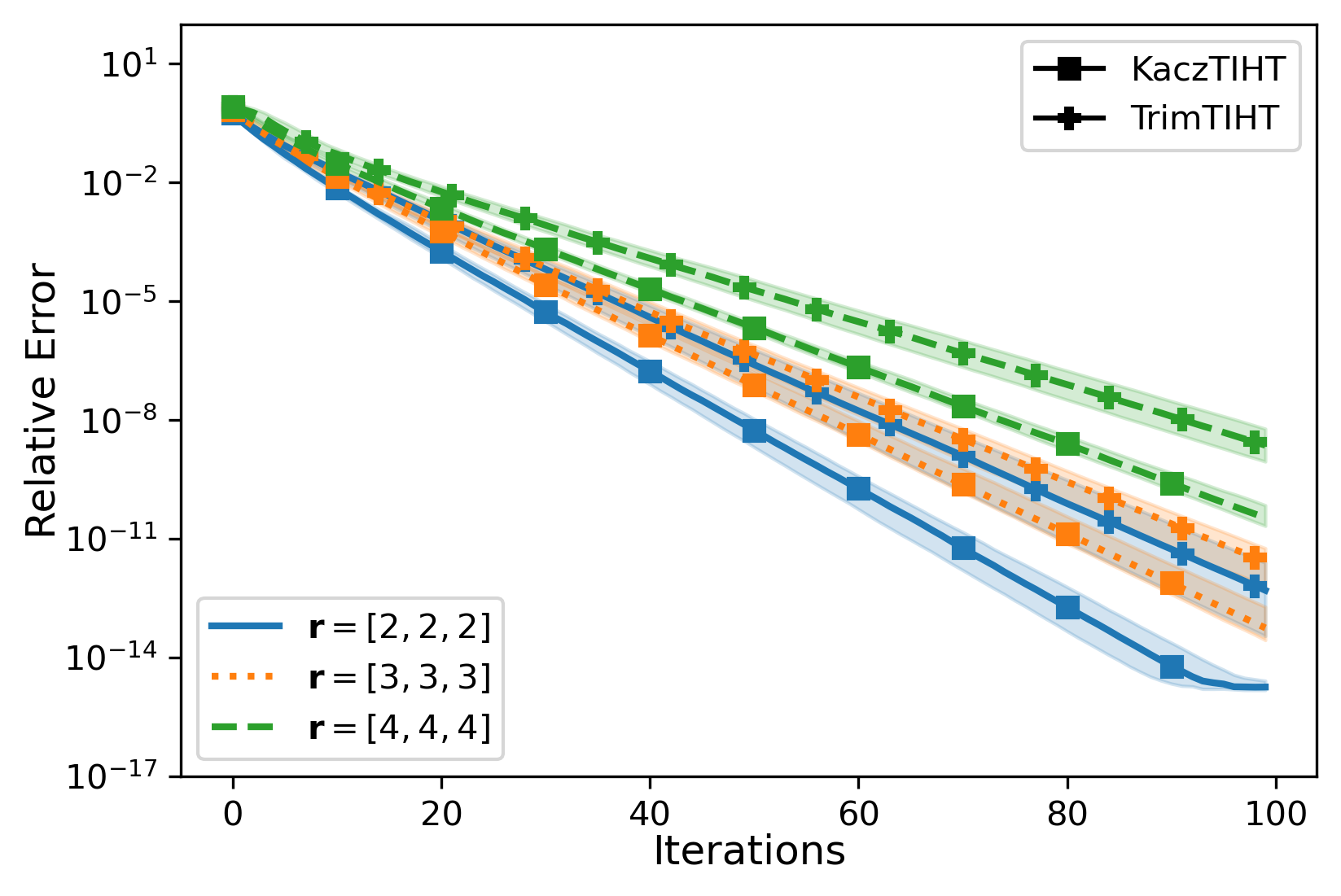}
     \label{fig:convergence_HOSVD}
     \caption{HOSVD rank ${\bf r}$, $m = 1\revised{0}50$ measurements}
 \end{subfigure}
  \hfill
\begin{subfigure}{0.45\textwidth}
     \includegraphics[width=\textwidth]{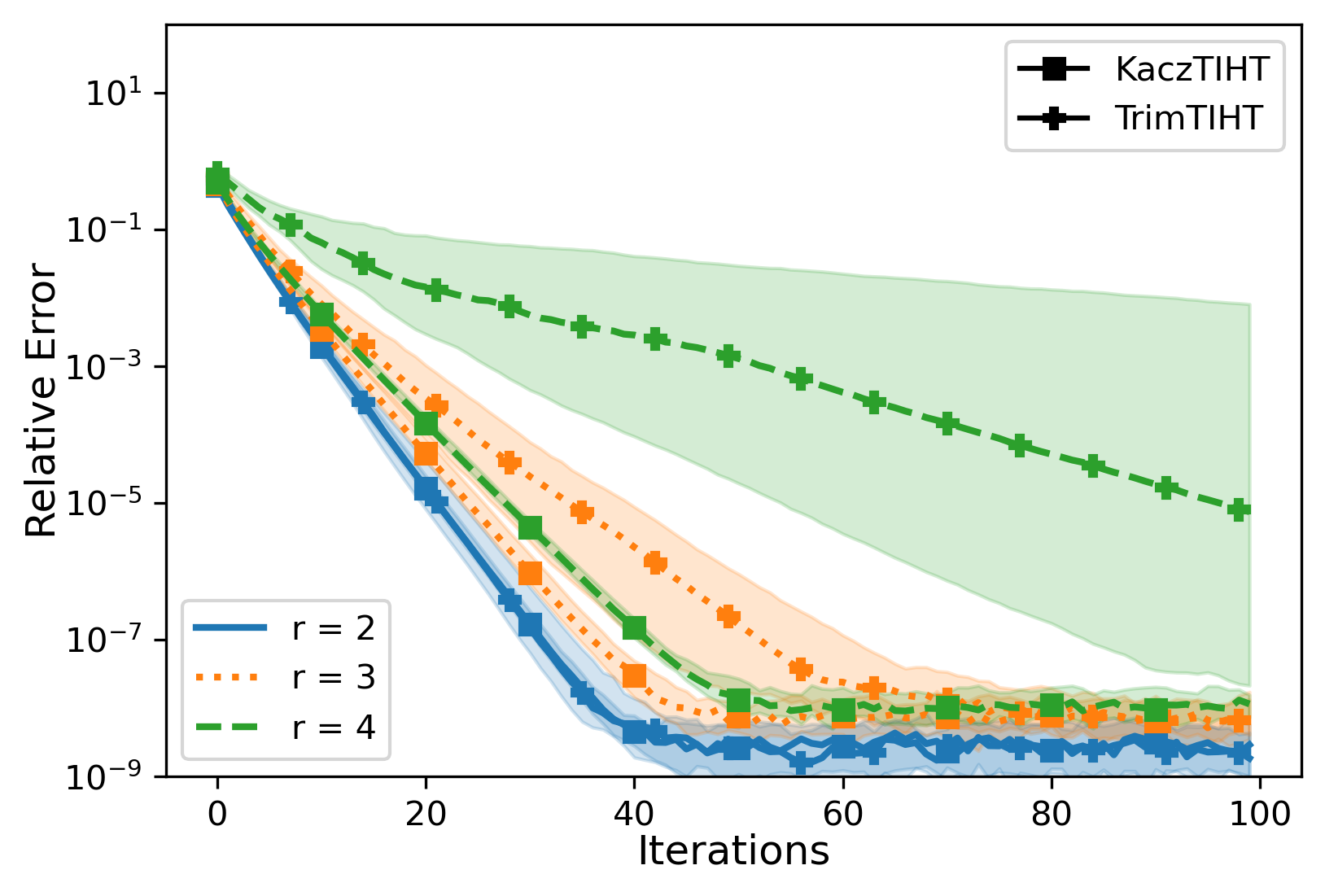}
     \label{fig:convergence_CP}
     \caption{CP rank $r$, $m = 1500$ measurements}

 \end{subfigure}
  \caption{\revised{Recovery \revised{of a $15 \times 15 \times 15$ tensor} from face-splitting measurements: Relative error dynamic during the recovery with KaczTIHT and TrimTIHT of  the tensors of different low rank. KaczTIHT tends to converge in fewer iterations, especially for higher rank tensors.
  The lines correspond to the median over $40$ sample runs, and the band represents the interquartile range.}}
       \label{img:face_split_convergence}
 \end{center}
\end{figure}

\revised{Then, we include the experiments on larger-scale tensors. Specifically, in Figure \ref{img:face_split_convergence_large} we consider the relative error dynamics of TrimTIHT at different level\revision{s} of compression, $\kappa = \frac{m}{n_1n_2\dots n_d}$, for low rank HOSVD tensors of dimensions $50\times 50 \times 50$ and $20 \times 20 \times 20 \times 20$. Here, we note that we were able to achieve recovery to a level of $10^{-3}$ from compression to the level of $3\%$ for the $3$-d tensor and $4\%$ for the $4$-d tensor within 150 iterations. Both KaczTIHT and TIHT algorithms diverged for this setup.}

\begin{figure}[!ht]
\begin{center}
 \begin{subfigure}{0.48\textwidth}
     \includegraphics[width=\textwidth]{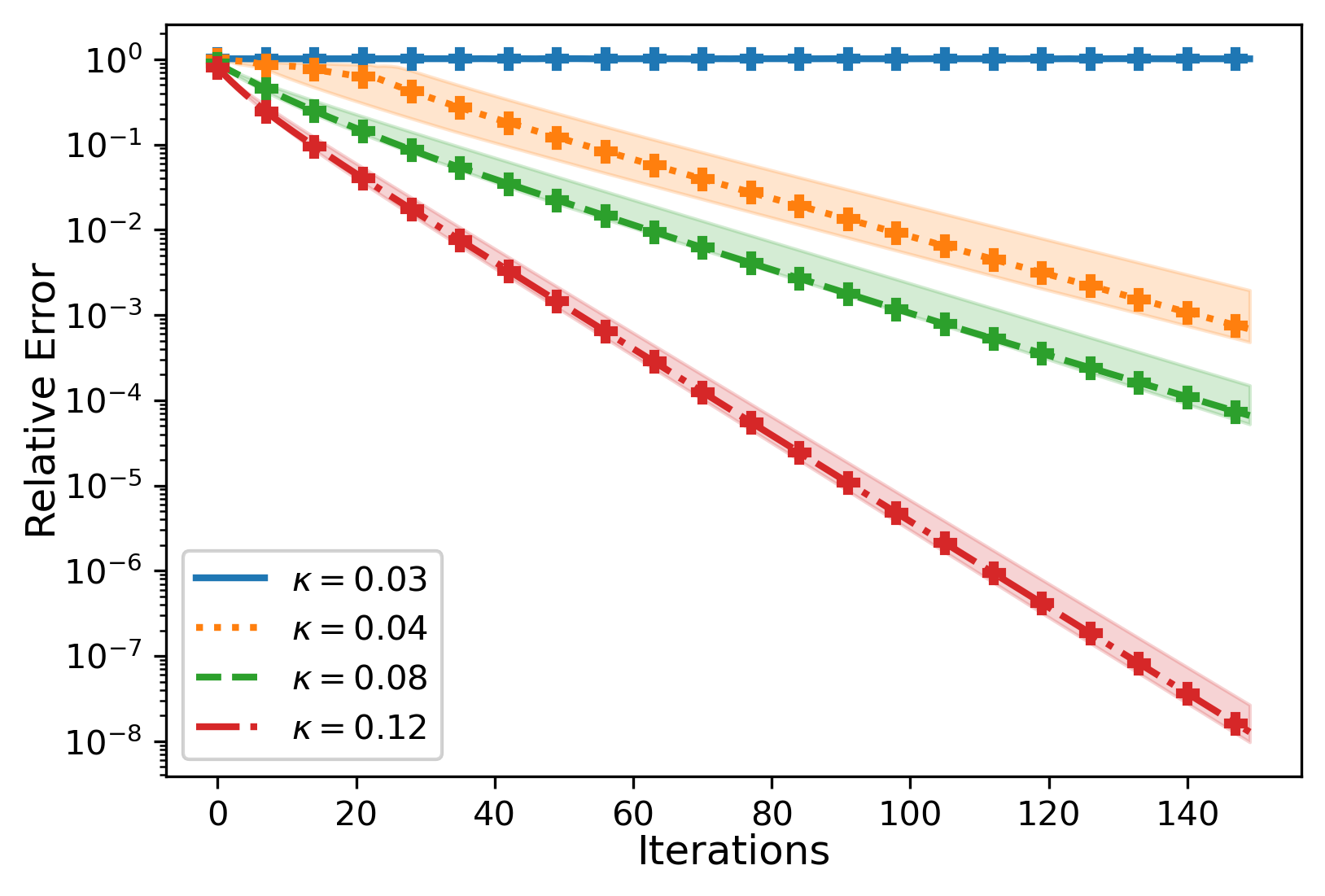}
     \label{fig:convergence_HOSVD_large}
     \caption{HOSVD rank $(2,2,2)$, $50 \times 50 \times 50$ tensor}
 \end{subfigure}
  \hfill
\begin{subfigure}{0.48\textwidth}
     \includegraphics[width=\textwidth]{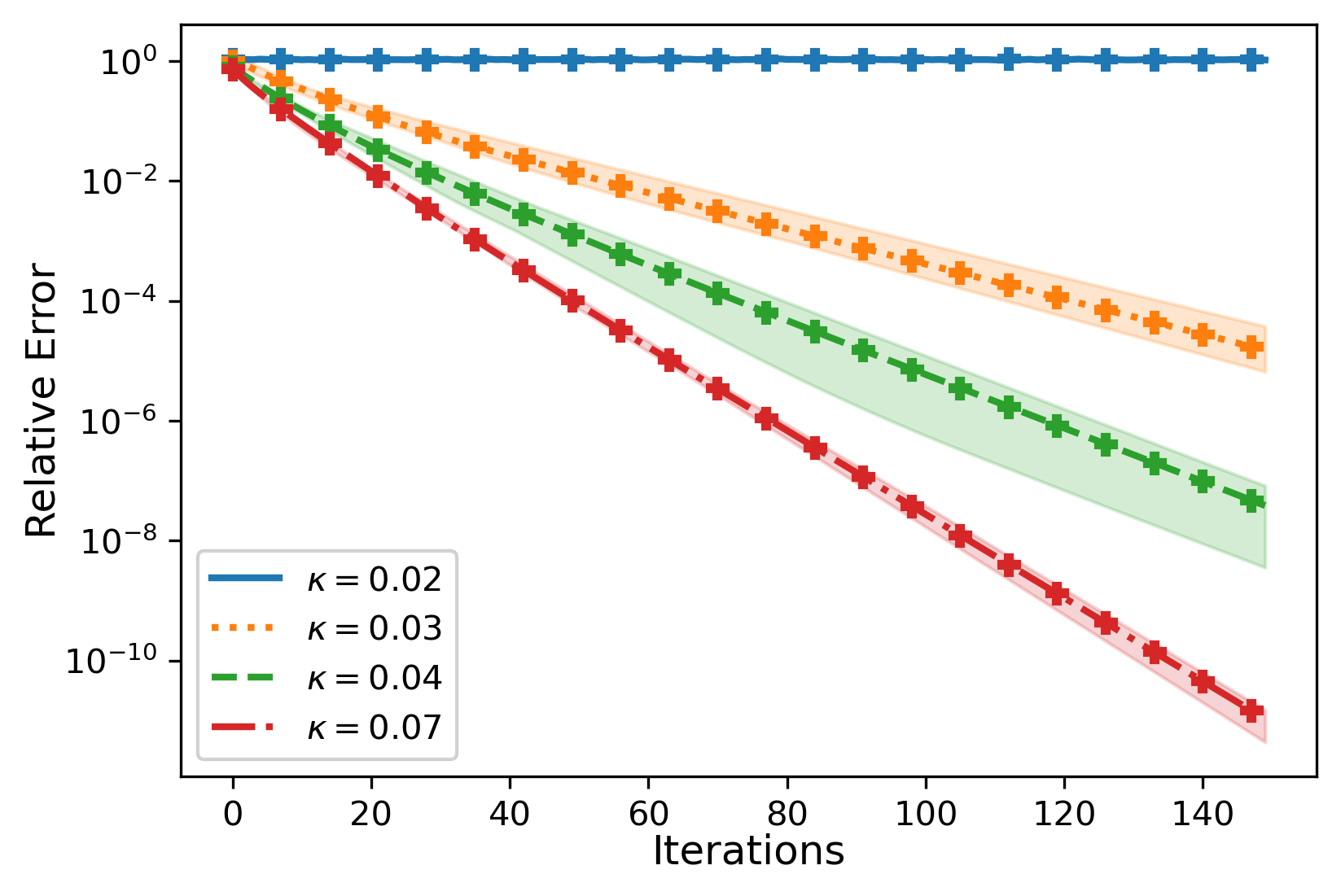}
     \label{fig:convergence_CP_large}
     \caption{HOSVD rank $(2,2,2,2)$, $20 \times 20 \times 20 \times 20$ tensor}

 \end{subfigure}
  \caption{\revised{Recovery from face-splitting product of Gaussian measurements: Relative error dynamic during the recovery with TrimTIHT of low rank HOSVD tensors at different compression levels $\kappa = \frac{m}{n^d}$. The lines correspond to the median over $10$ sample runs, and the band represents the $25\%$ interquartile range. Both KaczTIHT and TIHT diverge at this level of compression.}}
       \label{img:face_split_convergence_large}
 \end{center}
\end{figure}

\subsection{Video Data}
In this experiment, we apply our proposed methods to compress and recover a real-world video dataset. We use the video of a flickering candle flame from the Dynamic Texture Toolbox (\url{http://www.vision.jhu.edu/code/}).  For more tractable computations, we consider only $10$ frames cropped to $40 \times 40$ pixels around the center of the image. We compress the video tensor using i.i.d. Gaussian and face-splitting product of i.i.d. Gaussian measurement ensembles of dimensions \revised{$\R^{4000 \times 16000}$ which corresponds to $25\%$ compression. The original tensor is approximately of HOSVD rank $(6,7,2)$, with a relative error of low-rank tensor fitting via HOOI iteration (implemented in the Tensorly package) of $0.027$.} Our goal is to recover this structure from data-oblivious measurements, which is a harder task than the HOSVD components recovery tailored for a particular dataset. Yet, with \revision{the} TrimTIHT and KaczTIHT methods, we approach a compatible $10^{-2}$ relative error for the recovery from the data-oblivious measurements, as shown in Figure~\ref{img:candle_video}.

Specifically, in Figure~\ref{img:candle_video}, we plot the relative error with iterations of KaczTIHT and TrimTIHT \revised{(Here, we chose $m_{trim}= 650$ for face-splitting measurements and $m_{trim}= 500$ for Gaussian measurements)} for recovery from Gaussian measurements and from face-splitting measurements. Note that the \emph{face-splitting measurements take \revised{$175$} times less memory than unstructured i.i.d. measurements} \revised{($4000\times(40+40+10) \text{ vs. } 4000\times(40\times40\times 10)$ bits)} to store, while allowing for essentially the same recovery properties. 

In conclusion, both proposed methods are able to recover the original video to almost the best low-rank approximation error. \revision{However}, TIHT immediately diverges for both types of measurements and is \revision{therefore} not shown in the figure. KaczTIHT converges faster than TrimTIHT, and this can be attributed to the relatively high rank of the problem.
\begin{figure}[!ht]
\begin{center}
 \begin{subfigure}{0.37\textwidth}

     \raisebox{0.4cm}{\includegraphics[width=\textwidth,keepaspectratio]{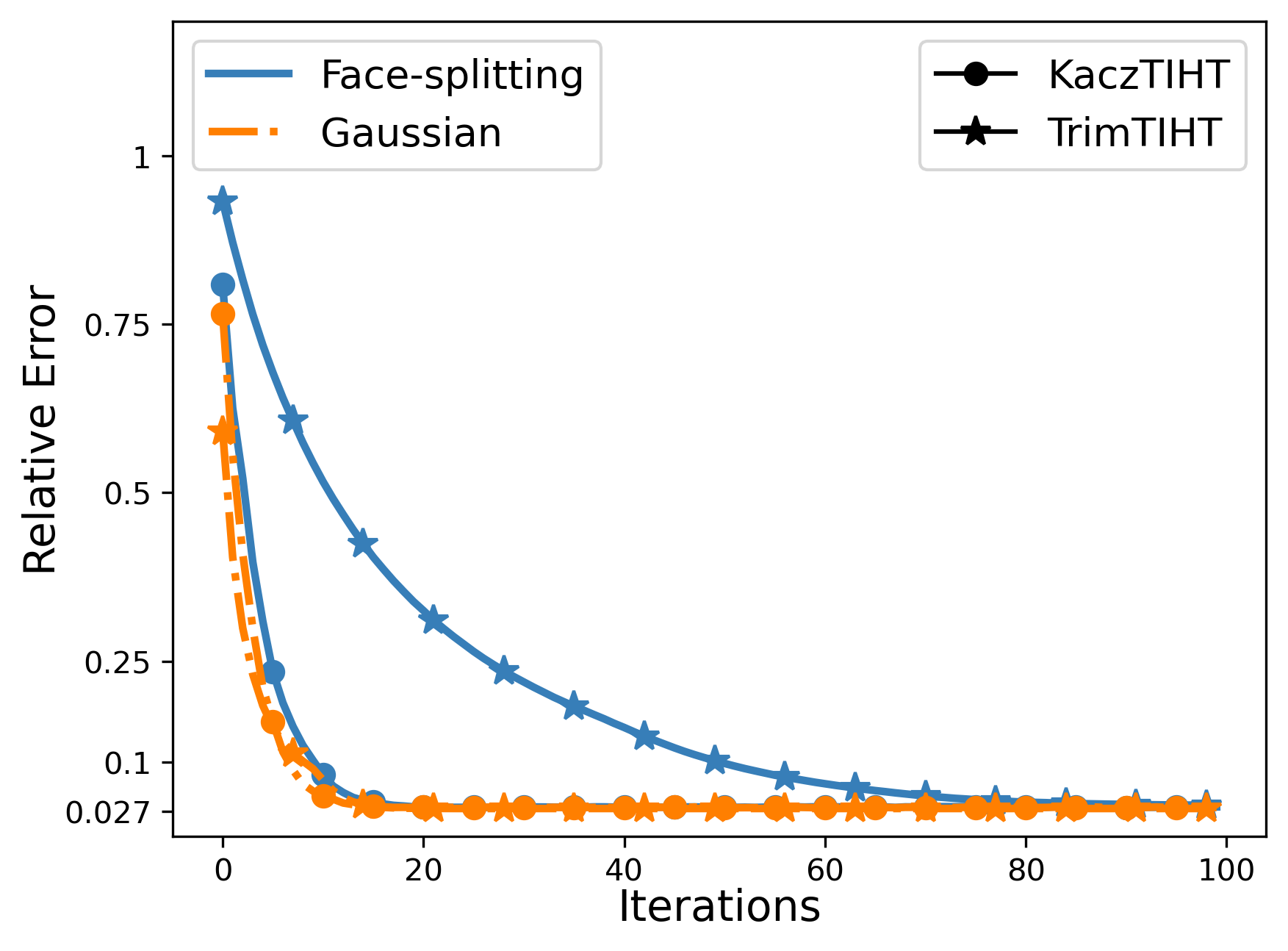}}
     \label{fig:candle_converge}
 \end{subfigure}
  \hfill
\begin{subfigure}{0.62\textwidth}    
\includegraphics[width=\textwidth,keepaspectratio]{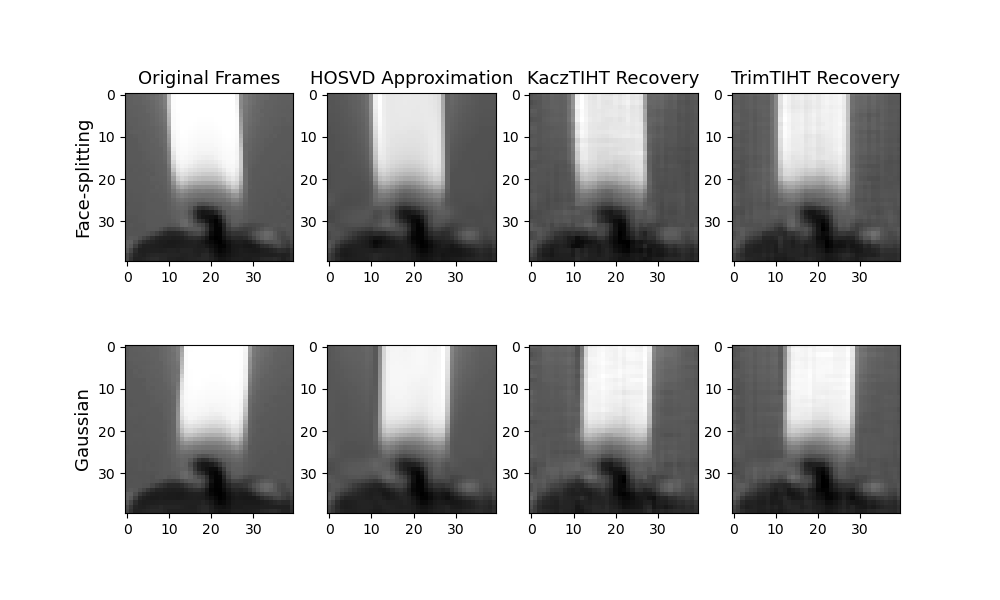}
     \label{fig:frames_rec}
 \end{subfigure}
 \end{center}
 \caption{\revised{(Left) Relative error of recovered Candle video dataset using KaczTIHT and TrimTIHT from Gaussian and face-splitting measurements. TIHT immediately diverges on the same setup. We aim to recover the $(6,7,2)$-rank approximation of the tensor from a compression to $25\%$ of its size. The achieved relative error is compatible with $0.027$ HOSVD fitting error, and KaczTIHT converges effectively as quickly on memory-efficient measurements (blue) as on independent Gaussian measurements (orange) while requiring $175$ times less memory. (Right) The first frame of the original video and its low rank approximation and recovered versions are presented in the first row, and the last frame and its low rank approximation and recovered images are depicted in the bottom row.}}
 \label{img:candle_video}
\end{figure}

\section{Conclusions and Future Directions}
\label{sec:conclusions}
In this paper, we \revised{develop methods for low-rank tensor recovery from the structured linear measurements}. We achieve this by proposing two efficient iterative hard thresholding-based methods -- TrimTIHT and KaczTIHT -- for low-rank tensor recovery, including HOSVD and CP ranks. We show that a class of practical, structured, and memory-efficient measurement operators, based on the face-splitting tensor product, fails to satisfy TensorRIP, a key property required to guarantee convergence for vanilla TIHT. Then, we analyze how local adaptive trimming of the measurements (the core idea of the proposed TrimTIHT method) can help in restoring TensorRIP-like guarantees for such database-friendly measurement operators. We also provide theoretical convergence guarantees for the proposed recovery methods and a suite of numerical experiments that showcase their efficacy \revised{over TIHT}, which is especially prominent in situations where we lack good TensorRIP properties.

There are multiple further questions stemming from this work. This includes providing theoretical evidence for the superior performance of the KaczTIHT method (this question seems to be unknown and interesting even for the sparse recovery counterpart of the method, see \cite{jeong2025linear}). Adaptive trimming in Kaczmarz methods is also known to improve robustness, particularly when \revision{handling corrupted measurements} \cite{haddock2022quantile,cheng2022block}. It would be natural to consider the effect of trimmed methods \revision{on} the robust low-rank recovery problem. \revised{Further, the theoretical guarantees provided for adaptive trimming hinge on the measurement matrix having independent rows. However, some measurement matrices with computationally desirable properties, such as mode-wise measurements, lack such independence guarantees. An interesting line of investigation would be to build on the proposed algorithms to facilitate efficient low-rank memory from such measurement operators.} Additional next directions include utilizing adaptive trimming within other low-rank tensor processing methods, many of which also have TensorRIP requirements, making them ineffective for recovery from database-friendly measurements. 

\section{Acknowledgments}
This work was partially supported by NSF DMS-2309685 and NSF DMS-2108479.
\printbibliography[
title={Bibliography}
]

\appendix

\end{document}